\documentclass[a4paper,reqno,10pt]{amsart}

\usepackage {amsmath,amssymb}
\usepackage{amsfonts}
\usepackage{amsthm}
\usepackage{tikz}
\usepackage[encapsulated]{CJK}
\usepackage{graphicx,color}
\usepackage [latin1]{inputenc}
\usepackage {times}
\usepackage {float}
\usepackage{hyperref}
\usepackage{amscd}
\usepackage[shortalphabetic,nobysame]{amsrefs}

\usetikzlibrary{matrix,arrows}

\numberwithin{equation}{section}
\numberwithin{figure}{section}

\DeclareMathOperator{\qqq}{q}
\DeclareMathOperator{\jj}{j}
\DeclareMathOperator{\trop}{B}
\DeclareMathOperator{\Ima}{Im}
\DeclareMathOperator{\cl}{cl}
\DeclareMathOperator{\rank}{r}

\DeclareMathOperator {\Star}{Star}

\DeclareMathOperator {\codim}{codim}

\DeclareMathOperator {\id}{id}

\newcommand{\mn}{\mathcal{M}_n}
\newcommand{\mnlab}{\mathcal{M}_n^{\text{\rm{lab}}}}
\newcommand{\mndelta}{\mathcal{M}_{n+\betrag{\Delta}}\times\R^r}
\newcommand{\bfan}{\mathcal{B}(M)}
\newcommand{\mf}{\mathcal{B}}

\newcommand{\diagonal}{\Delta_{B(M)}}
\newcommand{\Zy}{Z}
\newcommand{\betrag}[1]{\left|#1\right|}
\newcommand{\R}{\mathbb R}

\newcommand{\Z}{\mathbb Z}
\newcommand{\T}{\mathbb T}
\newcommand{\F}{\mathcal F}
\newcommand{\G}{\mathcal G}
\newcommand{\curlyx}{\mathcal{X}}
\newcommand{\curlyc}{\mathcal{C}}

\newcommand{\curlyb}{\mathcal{B}}

\newcommand{\diagM}{\Delta_M}
\newcommand {\PP}{{\mathbb P}}

\newcommand {\dcup}{\hspace{-0.7ex}
                      \begin{array}{c}
                        \\[-3.12ex]
                        \cdot\\[-2.3ex]
                        \cup
                      \end{array}
                      \hspace{-0.7ex}
                   }

\newcommand{\arxiv}[1]{%
  \href{http://arxiv.org/abs/#1}{arxiv:#1}%
}

\newtheorem {theorem}{Theorem}[section]
\newtheorem {proposition}[theorem]{Proposition}
\newtheorem {lemma}[theorem]{Lemma}
\newtheorem {defthm}[theorem]{Definition and Theorem}
\newtheorem {corollary}[theorem]{Corollary}
\theoremstyle {definition}
\newtheorem {definition}[theorem]{Definition}
\newtheorem {example}[theorem]{Example}
\theoremstyle {remark}
\newtheorem {remark}[theorem]{Remark}
\newtheorem {convention}[theorem]{Convention}
\newtheorem*{acknowledgement}{Acknowledgement}

\begin{document}

\title{The diagonal of tropical matroid varieties and cycle intersections}

\author {Georges Francois}
\address {Georges Francois, Fachbereich Mathematik, Technische Universit\"{a}t
  Kaiserslautern, Postfach 3049, 67653 Kaiserslautern, Germany}
\email {gfrancois@email.lu}

\author {Johannes Rau}
\address {Johannes Rau, Section de math\'ematiques, Universit\'e de Gen\`eve,
	Case postale 64, CH-1211 Gen\`eve 4, Switzerland}
\email {johannes.rau@unige.ch}

\begin{abstract}
We define an intersection product of tropical cycles on matroid varieties 
(via cutting out the diagonal) and show that it is well-behaved. 
In particular, this enables us to intersect cycles on moduli spaces of 
tropical rational marked curves $\mn$ and $\mnlab(\Delta, \R^r)$. 
This intersection product can be extended to smooth varieties
(whose local models are matroid varieties).
We also study pull-backs of cycles and rational equivalence. 
\end{abstract}

\maketitle

\section{Introduction}
For each loopfree matroid $M$ with ground set $E$ 
there is an associated tropical cycle $\trop(M) \subset \R^E$. It is
a fan with lineality space $\R \cdot (1,\ldots,1)$ whose dimension
is equal to the rank of the matroid. These objects, which we call
\emph{matroid varieties} here, have been studied, among others, by
Sturmfels, Ardila, Klivans, Speyer and Feichtner
\cites{Sturmfels,ardila,Speyer,FS}.
Matroid varieties 
generalise the tropicalisations of classical linear spaces and can therefore be considered as tropical linear spaces. 
In particular, they
are natural candidates for being the local building blocks 
of \emph{smooth} tropical varieties.
Therefore, it has been expected that on such spaces a well-behaved intersection
product of tropical subcycles exists. 
The aim of this article is to construct this intersection product, to analyse
some of its properties and to relate it to other notions such as rational equivalence.
As in \cite[section 9]{AR} for $\R^n$ and in \cite[section 1]{lars} for $L_k^n$, our 
construction is based on finding rational functions on the product $\trop(M)\times\trop(M)$ which cut out the diagonal $\Delta_{\trop(M)}$. 

An intersection product on matroid varieties was presented before by Kristin Shaw (cf.\ \cite{kristin}). 
Her alternative approach uses tropical modifications 
to give a recursive definition.
In particular, the observation that elementary quotients
of matroids correspond to tropical modifications is due to her.
In theorem \ref{compkristin}, we show that both definitions of intersection
products on matroid varieties agree --- therefore, the advantages of
both approaches can be combined. 

In section 3 we show that any matroid variety
contained in a second one can be cut out from the second one by 
explicitly given rational functions. 
Applying this in section 4 to the diagonal of a matroid variety 
sitting in the cartesian product
enables us to construct an intersection product of cycles on $\trop(M)$ having the usual properties.

In our terminology, $\trop(M)$ denotes a 
``tropical affine cone'' with lineality space $L = \R\cdot(1,\ldots,1)$. 
A priori, our intersection product is defined on this cycle. 
Therefore, section 5 is devoted to carrying over this intersection product
to the projectivisation $\trop(M)/L$. This is a mainly technical task.
In section 6, we give a definition of smooth tropical varieties (whose
local models are $\trop(M)/L$) and extend the intersection product to this case.

As an application, in section 7 we identify (on the level of tropical varieties) the 
moduli spaces of tropical rational curves $\mn$ and $\mnlab(\Delta, \R^r)$
with matroid varieties obtained from the complete graph $K_{n-1}$ and hence get
an intersection product on these spaces. 
Finally, in sections 8 resp.\ 9 we study pull-backs of cycles
resp.\ rational equivalence.

We are grateful to Kristin Shaw, Federico Ardila and Andreas Gathmann for many helpful
discussions and comments.

\begin{acknowledgement}
Georges Fran\c{c}ois is supported by the Fonds national de la Recherche (FNR), Luxembourg.
\end{acknowledgement}

\begin{convention}
In the following, unless explicitly told otherwise, all matroids are assumed to be loopfree (that means that each element of the ground set has rank 1). 
\end{convention}

\section{Preliminaries}

We start with recalling the definition of a matroid variety.
We state a few general results about matroid varieties
which are needed in the following. 

Let $M=(E,\curlyb)$ be a loopfree matroid of rank $\rank(M)$
with ground set $E=\{1,\ldots,n\}$. 
It defines a tropical fan cycle $\trop(M)$ of dimension $\rank(M)$ in $\R^n$
whose support set can be described as follows:
To each point $p \in \R^n$ one can associate a matroid $M_p$ whose
bases are the $p$-minimum bases of $M$ (where the $p$-weight of
a basis $B$ is $\sum_{i \in B} p_i$). The point $p$ lies in the support
of $\trop(M)$ if and only if the matroid $M_p$ is (still) loopfree.
As tropical cycle, $\trop(M)$ can be obtained from the 
unimodular fan $\bfan$ consisting of the cones 
\[\langle \F \rangle:=\left\{\sum_{i=1}^{p}\lambda_i\cdot V_{F_i}: \lambda_1,\ldots,\lambda_{p-1}\geq 0, \lambda_p\in\R\right\},\] 
where $\F =(\emptyset\subsetneq F_1\subsetneq\ldots\subsetneq F_{p-1}\subsetneq F_p=E)$ is a chain of flats in $M$, and $V_F=-\sum_{i\in F} e_i$ denotes the vector corresponding to the flat $F$. 
Here $\{e_1,\ldots,e_n\}$ denotes the standard basis of $\R^n$. 
Equipped with trivial weights $1$ for each facet,
$\bfan$ is balanced (as defined for example in \cite[2.6]{AR}).
We call the resulting tropical cycle $\trop(M)$
the \emph{matroid variety} corresponding to $M$.
Following \cite{ardila}, we call $\bfan$ the \emph{fine subdivision} of $\trop(M)$.
Note that, by definition, $\trop(M)$ has lineality space $\R\cdot (e_1+\ldots+e_n)$.

The following picture shows the fine subdivision of the matroid variety $\trop(U_{3,4})$ (modulo its lineality space $\R\cdot (1,1,1,1)$). Here $U_{3,4}$ denotes the uniform matroid of rank $3$ on the set $N:=\{1,2,3,4\}$ whose bases are the $3$-subsets of $N$ (cf.\ example \ref{lnk}). The maximal cones of $\mf(U_{3,4})$ are of the form $\langle \emptyset\subsetneq\{i\}\subsetneq \{i,j\}\subsetneq\{1,2,3,4\}\rangle$.

\begin{center}
\includegraphics[scale=0.30]{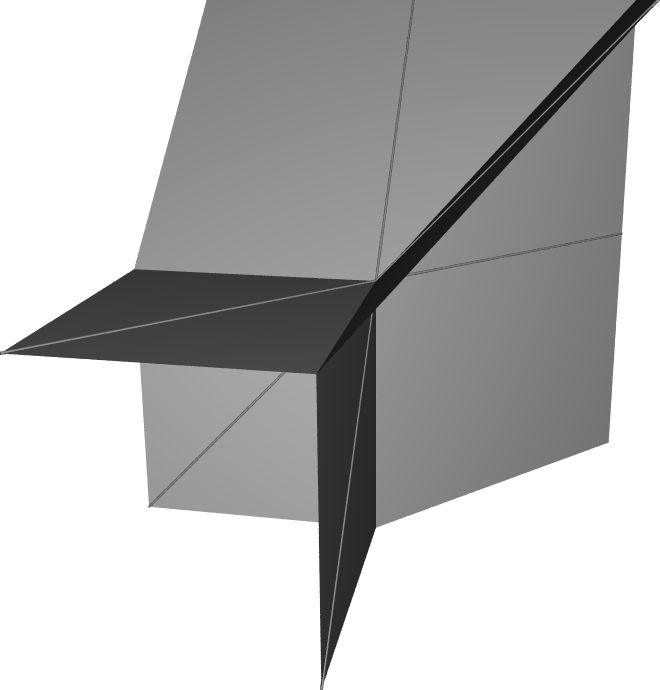}
\end{center}

Recall that the direct sum $M\oplus N$ of two matroids $M$ and $N$ 
is the matroid whose ground set is the disjoint union $E(M)\dcup E(N)$ and whose set of bases is the set $\{B_M\dcup B_N: B_M, B_N \text{ bases of } M, N \text{ respectively}\}$. 

\begin{lemma} \label{directsum}
Let $M, N$ be matroids. Then the two tropical cycles $\trop(M\oplus N)$ and $\trop(M)\times\trop(N)$ are equal.
\end{lemma}

\begin{proof}
The equality of the support sets follows from the equality of matroids \[(M\oplus N)_{(p,q)}=M_p\oplus N_q\] (see also \cite[proposition 2.5]{Speyer}).
As all occurring weights are $1$, that also shows the equality
of the cycles.
\end{proof}

Our next remark about matroid varieties concerns their local structure.
Let us fix our terminology first.

Let $X$ be a tropical cycle in a vector space $V$ and let $p$ be a point
in $X$. We define $\Star_X(p)$ to be the cycle (in $V$) associated to
$\Star_{\curlyx}(\{p\})$ (cf.\ \cite[section 1.2.3]{disshannes}), where
$\curlyx$ is a polyhedral structure of $X$ containing the cell $\{p\}$.
In other words, $\Star_X(p)$ is the fan cycle containing all vectors $v$ such that
$p + \epsilon v \in |X|$ for sufficiently small (positive) $\epsilon$
(with inherited weights).

\begin{lemma}
  \label{localmatroidfans}
  Let $\trop(M)$ be a matroid variety and $p$ a point in $\trop(M)$.
  Then we have
  \[
    \Star_{\trop(M)}(p)=\trop(M_p),
  \] 
  where $M_p$ is the matroid whose bases are the $p$-minimum bases of $M$.
\end{lemma}

\begin{proof}
  The statement follows from the identity
  $M_{p + \epsilon v} = (M_p)_v$ for any vector $v$
  and sufficiently small $\epsilon$.
\end{proof}

Any matroid $M$ can be decomposed into a direct sum 
$M = M_1 \oplus \cdots \oplus M_k$ of connected
submatroids which is unique up to reordering (cf.\ \cite[corollary 4.2.13]{matroidtheory}). 
It follows from lemma \ref{directsum} that the (maximal) 
lineality space of $\trop(M)$ is of dimension at least the number of
connected components $k$.
Here, 
a lineality space $L$ of a tropical cycle $C \subseteq \R^n$
is a subspace of $\R^n$ such that $C$ is invariant under translations
by vectors in $L$ (see section \ref{dividingout} for further terminology).
The next lemma states that equality holds.

\begin{lemma}
\label{connectedcomponents}
Let $M$ be a matroid on the ground set $E$
and let $\trop(M)$ be the corresponding matroid variety.
Let $L$ be its maximal lineality space. Then the equation
\[
  \dim(L) = \text{number of connected components of }M
\]
holds.
In particular, if $M$ is connected, then $L$ is just spanned by $(1,\ldots,1)$.
\end{lemma}

\begin{proof}
  Obviously, it suffices to show that $M$ is disconnected if
  $\dim(L) > 1$. So let us assume that $L$ is more than the span of $(1,\ldots,1)$.
  Then $L$ must contain some vector $V_S$ with $\emptyset \subsetneq 
  S \subsetneq E$. The fact that $V_S$ is contained in the lineality space
  of $\trop(M)$ means that $M_{\lambda V_S}$ stays the same for all
  $\lambda \in \R$; in particular, $M_{V_S} = M$. 
  Hence all bases of $M$ have the same $V_S$-weight, i.e.\ have the same 
  number of elements in $S$ (resp.\ $E \setminus S$). This shows that 
  $S$ is a separator (i.e.\ a union of connected components).
\end{proof}

We finish this section by showing that a matroid 
variety $\trop(M)$ is always irreducible (i.e.\ any subcycle 
$X$ of $\trop(M)$ of the same dimension is $X = m \cdot \trop(M)$
for some integer $m$). 
Moreover, by lemma \ref{localmatroidfans} this implies that
a matroid variety $\trop(M)$ is locally irreducible 
(i.e.\ for every point $p$ in $\trop(M)$, the local fan
$\Star_{\trop(M)}(p)$ is also irreducible, 
cf.\ \cite{disshannes}*{definition 1.2.27}).
To show irreducibility, let us recall that for
any cycle $X \subseteq \R^n$ of pure dimension $k$ 
the \emph{projective degree} of $X$ is given
by
\[
  \deg(X) := \deg(\max\{0,x_1, \ldots, x_n\}^{n-k} \cdot X).
\]
It follows easily from \cite{Speyer}*{section 3} that each matroid
variety has projective degree $1$ (moreover, Fink shows 
a converse statement
in \cite{Fink}*{theorem 6.5}). As in the classical case, this implies irreducibility,
though this implication is not quite as trivial as 
$\trop(M)$ could be split into cycles with possibly negative weights.

\begin{lemma}
  Each matroid variety $\trop(M)$ is irreducible, and therefore,
  by lemma \ref{localmatroidfans}, also locally irreducible.
\end{lemma}

\begin{proof}
  Let $k$ be the dimension of $\trop(M)$ and let $H$ be a translation of 
  \[
    \max\{0,x_1, \ldots, x_n\}^{n-k} \cdot \R^n.
  \]
  As $\deg(\trop(M)) = 1$ and both $\trop(M)$ and $H$ carry only positive
  weights, it is not hard to see that for a generic 
  point $p$ in $\trop(M)$, there is a suitable translation $H$ such that 
  $\{p\} = |\trop(M)| \cap |H|$.
  Now let $X$ be a subcycle of $\trop(M)$ of the same dimension. 
  For $p$ and $H$ as before,
  we must have $H \cdot X = \deg(X) \cdot \{p\}$, 
  and therefore $X = \deg(X) \cdot \trop(M)$.
\end{proof}

\section{Matroid quotients and rational functions}

Let us fix a set $E$ with $n$ elements and let $O$ be the (trivial) matroid of rank $n$ on $E$ (i.e.\ the only basis is given by the whole set $E$).
Following the construction of \cite{ardila} described earlier, 
we get a subdivision $\mf(O)$ of $\R^n$, with
minimal cone $\R\cdot(1,\ldots,1)$, which
is also called \emph{braid arrangement}.
By construction, any other matroid $M$ on $n$ elements produces a subfan
$\mf(M)$ of this subdivision $\mf(O)$. 
For two matroids $M$ and $N$, we conclude
\[
  |\mf(N)| \subseteq |\mf(M)| \;\Leftrightarrow\;
    \mf(N) \subseteq \mf(M) \;\Leftrightarrow\;
    \{\text{flats of } N\} \subseteq \{\text{flats of } M\}.
\]
In the following, we just write $\trop(N) \subseteq \trop(M)$ in this situation.
The last equivalence has the following consequences on the rank functions of $M$ and $N$.

\begin{lemma} \label{equalities}
  Let $M$ and $N$ be matroids of rank $r$ resp.\ $s$ such that
  $\trop(N) \subseteq \trop(M)$. Let $A \subseteq B$ be arbitrary subsets
  of $E$. Then the equation
  \[
    \rank_M(A)-\rank_N(A) \leq \rank_M(B)-\rank_N(B)
  \]
  holds. Plugging in $A=\emptyset$ and $B = E$, we obtain
  \[
    \rank_N(A) \leq \rank_M(A) \leq \rank_N(A)+r-s.
  \]
\end{lemma}

\begin{proof}
  As $\cl_M(A) \subseteq \cl_N(A)$ for any set $A$, 
  we can assume that $A$ and $B$ are closed in $M$. By induction,
  we can also assume $\rank_M(B) - \rank_M(A) =1$, i.e.\
  $B = \cl_M(A \cup x)$ for an element 
  $x \in B \setminus A$. It follows $\cl_N(B) = \cl_N(\cl_M(A \cup x))
  = \cl_N(A \cup x)$, 
  i.e.\ $\rank_N(B) - \rank_N(A) \leq 1$, 
  which proves the claim.
  Another proof is contained in \cite[proposition 7.3.6]{matroidtheory}.
\end{proof}

We will now see that there is a notion in matroid theory which
captures containment of matroid varieties. This notion
is based on the following standard constructions for matroids.

Let $Q$ be a matroid on the set $E \dcup R$. Then the \emph{deletion}
$Q \setminus R$ is the matroid on $E$ given by the rank function
\[
  \rank_{Q \setminus R}(A) = \rank_Q(A),
\]
whereas the \emph{contraction} $Q/R$ is the (potentially not loopfree) 
matroid on $E$ given by
\[
  \rank_{Q / R}(A) = \rank_Q(A \cup R) - \rank_Q(R).
\]
Note that $Q/R$ is loopfree if and only if $R$ is a flat in $Q$.
The next definition (following \cite[section 7.3]{matroidtheory})
combines both operations.

\begin{definition}
  Let $M$ and $N$ be matroids of rank $r$ resp.\ $s$ 
  on the same ground set $E$. We call $N$ a \emph{quotient} of
  $M$ if there exists a third matroid $Q$ on the ground set $E \dcup R$ such 
  that $M = Q \setminus R$ and $N = Q / R$.
  In this case, we have $r - s = r_Q(E) + \rank_Q(R) - r_Q(Q)$.
  Furthermore, if $r-s=1$, we call $N$ an \emph{elementary quotient}
  of $M$.
\end{definition}

Now, in fact, containment of matroid varieties is related to quotients as follows.

\begin{proposition} \label{quotient}
  The matroid variety $\trop(N)$ is a subcycle of $\trop(M)$ if and only if $N$ is a 
  quotient of $M$.
\end{proposition}

A quite lengthy proof can be found in \cite[proposition 7.3.6]{matroidtheory}. To be self-contained, we give a short proof here. 
We use the following criterion for rank functions.

\begin{theorem}[\cite{matroidtheory} theorem 1.4.14] \label{criterionrank}
  Let $\rank$ be an integer valued function on the set of subsets of $E$. Then
  $\rank$ is the
  rank function of a matroid if and only it satisfies the following properties.
  \begin{enumerate}
    \item $\rank(\emptyset) = 0$.
    \item If $A \subseteq E$ and $x \in E$, then $\rank(A) \leq \rank(A \cup x)
          \leq \rank(A) + 1$.
    \item If $A \subseteq E$ and $x,y \in E$ such that $\rank(A \cup x) =
          \rank(A \cup y) = r(A)$, then $\rank(A \cup x \cup y) = r(A)$.
  \end{enumerate}
\end{theorem}

\begin{proof}[Proof of proposition \ref{quotient}] 
  If $N$ is a quotient of $M$, then it follows straight from the definitions
  that
  any flat of $N$ is also closed in $M$. This proves one implication.
  
  For the other direction, let us assume $\trop(N) \subseteq \trop(M)$.
  First, we fix a set $R$ with $r-s$ elements. We define a matroid $Q$
  on $E \dcup R$ by assigning to each subset $I \dcup J \subseteq E \dcup R$
  the rank 
  \begin{equation} \label{eqmin} 
    \rank_Q(I \dcup J) = \min\{\rank_M(I) + |J|, \rank_N(I) + r - s\}.
  \end{equation}
  Using the inequalities of lemma \ref{equalities} and plugging in
  $I \dcup \emptyset$, $I \dcup R$ and $\emptyset \dcup R$, we
  see that   indeed $Q \setminus  R = M$ and $Q/R = N$. 
  
  It remains to check, by using the criteria of theorem \ref{criterionrank}, 
  that $\rank_Q$ is indeed a rank function. The first criterion 
  is trivial,
  the second one follows from the corresponding property of $\rank_M$ and
  $\rank_N$.
  As for the third criterion,
  for a given $A = I \dcup J$, note that if adding an element $x$ does not
  increase the first term of the minimum in equation \eqref{eqmin}, 
  then it does not increase the 
  second term either, as we have $\rank_N(I \cup x) - \rank_N(I) \leq 
  \rank_M(I \cup x) - \rank_M(I)$ by lemma \ref{equalities}. So the third 
  property follows from the respective property of $\rank_N$ 
  (if the minimum in equation \eqref{eqmin}
  is attained in the second term) and $\rank_M$ (otherwise). This finishes the proof.
\end{proof}

\begin{remark}
\label{Q}
Note that the matroid $Q$ we constructed is minimal in the 
following sense. It is loopfree, $R$ is independent 
and closed in $Q$ and $\rank(Q)=\rank(M)$
(cf.\ \cite{matroidtheory}*{lemma 7.3.3}). 
\end{remark}

Let us use proposition \ref{quotient} now. We start again with two matroids $N$ and $M$ such that $\trop(N) \subseteq \trop(M)$. Let $Q$ be the matroid we constructed
in the previous proof and assume $R = \{1, \ldots, r-s\}$. Instead of deleting (or contracting) the whole set $R$ we might do the following: For given $i,j \geq 0$ with $i+j \leq r-s$, we can define the matroid
\[
  Q \setminus i \;/ j := (Q \setminus \{1,..i\})/\{i+1, \ldots, i+j\}
\]
on the ground set $E \dcup \{i+j+1, \ldots, r-s\}$. Of course,
as $Q$ is symmetric in $R$, we could as well have chosen any other subsets
of $R$ with $i$ resp.\ $j$ elements. Particularly interesting are the
matroids $M_i := Q \setminus i \;/ (r-s)-i$ with ground set $E$. We see
directly from the definition that $M_0 = N$, $M_{r-s} = M$ and 
$\rank_{M_i}(E) = s + i$. Moreover, one can easily check that $\trop(M_i) \subseteq \trop(M_j)$ holds for all $i \leq j$ (either by computing the flats
using for example \cite[proposition 3.3.1]{matroidtheory} or by noting that
$Q \setminus i \;/ (r-s) - j$ is a matroid that makes $M_i$ a quotient 
of $M_j$). The rank function of $M_i$ is given by
\[
  \rank_{M_i}(A):=\min\{\rank_N(A)+i,\ \rank_M(A)\}.
\]
Summarising, we get the following statement.

\begin{corollary}
\label{zwischenmatroid}
Let $M$ and $N$ be matroids of rank $r$ resp.\ $s$ such that $\trop(N)\subseteq\trop(M)$. 
Then there exists a sequence of matroids $M_i$ with the properties
$M_0=N$, $M_{r-s}=M$, $\rank_{M_i}(E)=\rank_{N}(E)+i$, and $\trop(M_{i})\subseteq\trop(M_{i+1})$.
\end{corollary}

\begin{remark}
  The matroid-theoretic counterpart of this statement can be found in
  \cite[proposition 7.3.5]{matroidtheory}. 
\end{remark}

The previous corollary suggests to study elementary quotients in more detail. 
Before we do that, let us have a look towards the geometric meaning of
deletions and contractions of matroids.

Let $Q$ be matroid on the set $E \dcup R$ and let $\trop(Q)$ be its matroid variety in $\R^{E \dcup R}$. 
Assume that $R$ is a flat of $Q$ (i.e.\ $Q/R$ is loopfree) and that there exists a basis $B$ of $Q$ such that $R\cap B=\emptyset$ (i.e.\ $\rank(Q)=\rank(Q\setminus R)$). From that, we construct two tropical cycles in $\R^E$.
First, the projection map $\pi_R : \R^{E \dcup R} \rightarrow \R^E$ 
produces the push-forward $(\pi_R)_*(\trop(Q))$ (for the definition of
push-forward, see e.g.\ \cite[definition 1.3.6]{disshannes}).
Second, we can take the closure of $\trop(Q)$ in $(\R \cup \{-\infty\})^{E \dcup R}$ and
perform the intersection $\overline{\trop(Q)} \cap (\R^E \times \{-\infty\}^R)$
with a coordinate plane at infinity.
In other words, we intersect $\trop(Q)$ with $\R^E \times \{-\lambda\}^R$,
where $\lambda$ is a large real number.
Let us denote the resulting set/cycle in $\R^E$ by $\trop(Q)^{\cap R}$.
Now, the following statement relates these geometric constructions to
the matroid-theoretic notions of contraction and deletion.

\begin{lemma} \label{geometric}
  With the notations and assumptions from above, we see that 
  the deletion of $R$ corresponds 
  to projecting, i.e.\
  \[
    \trop(Q \setminus R) = (\pi_R)_*(\trop(Q)),
  \]
  and the contraction of $R$ corresponds to intersecting with the
  appropriate coordinate hyperplane at infinity, i.e.\
  \[
    \trop(Q / R) = \trop(Q)^{\cap R}.
  \]
  Moreover, the map $\pi_R : \trop(Q) \rightarrow \trop(Q \setminus R)$ is
  generically one-to-one. 
\end{lemma}

\begin{remark}
  Note that, if the matroid $Q$ is realisable, then the analogue
  classical statement is well-known. See also \cite{kristin}*{section 2}
  for similar statements.
\end{remark}

\begin{proof}[Proof of lemma \ref{geometric}]

  For the first equation, let $\sigma$ be a cone in $\mf(Q)$ and let
  $\F= (\emptyset \subsetneq F_1\subsetneq \ldots \subsetneq F_r)$ be 
  the corresponding chain of flats in $Q$. Then the projection
  of $\sigma$ along $\pi_R$ is obviously given by the chain $\G$
  with $G_i = F_i \setminus R$, which is a chain of flats in $Q \setminus R$.
  Hence $\pi_R(\sigma)$ is a cone in $\mf(Q \setminus R)$. Furthermore,
  for any maximal chain $\G$ of flats in $Q \setminus R$, there is exactly
  one ``lifted'' chain $\F$, namely given by $F_i = \cl_Q(G_i)$.
  Note that $\F$ is maximal as we assume that $Q$ and $Q\setminus R$ have the same rank.
  Thus for each
  maximal cone of $\mf(Q \setminus R)$ there is exactly one maximal cone in
  $\mf(Q)$ mapping to it (with trivial lattice index) and $\pi_R$ is one-to-one
  over points in the interior of maximal cones.
  
  For the second equation, we have the following chain of equivalences.
  \begin{eqnarray*}
    & & p \in \trop(Q)^{\cap R} 
      \;\; \Leftrightarrow \;\; 
      (p, -\lambda, \ldots, -\lambda) \in \trop(Q) \text{ for large $\lambda$} \\
    & \Leftrightarrow & 
      Q_{(p, -\lambda, \ldots, -\lambda)} \text{ loopfree for large $\lambda$} \\
    & \Leftrightarrow & 
      \text{for all $a \in E$ there exists a basis $B$ of $Q$
      such that $a \in B$, $|B \cap R| = \rank_Q(R)$  }   \\
    & & \text{and $B \cap E$ is $p$-minimal     } \\
    & \Leftrightarrow & 
      \text{for all $a \in E$ there exists a $p$-minimal basis $B'$ of $Q/R$
      such that $a \in B'$} \\
    & \Leftrightarrow & 
      (Q/R)_p \text{ loopfree}
      \;\; \Leftrightarrow \;\; p \in \trop(Q/R)
  \end{eqnarray*}
  In the middle step we use that bases $B'$ of $Q/R$ are exactly obtained
  as $B' = B \cap E$, where $B$ is basis of $Q$ with $|B \cap R| = \rank_Q(R)$.
\end{proof}

We now turn to the case of elementary quotients. Using the above
description we will see that they are in fact related to modifications
in the sense of Mikhalkin (cf.\ \cite{MAppl}). 
This observation was first made by Kristin Shaw (cf.\ 
\cite{kristin}*{proposition 2.24}, to which we also refer for
further details).

\begin{proposition} \label{modification}
  Let $M$ and $N$ be matroids of rank $r$ resp.\ $r-1$ such that   
  $\trop(N)\subseteq\trop(M)$. Let $Q$ be the matroid on $E \dcup \{e\}$
  constructed in proposition \ref{quotient} with $Q \setminus e = M$
  and $Q/e = N$.
  Then $\trop(Q)$ is a modification of $\trop(M)$ along the divisor
  $\trop(N)$ in the sense of \cite{MAppl}*{section 3.3}. 
  The modification function $\varphi$ on $\trop(M)$ is given by its
  values on the vectors $V_F$, $F$ closed in $M$ as
  \[
    \varphi(V_F)= \rank_N(F) - \rank_M(F).
  \]
  In particular, the divisor of $\varphi$ is 
  $\varphi \cdot \trop(M) = \trop(N)$.
\end{proposition}

\begin{proof}
  Let $\varphi$ be as defined above. According to our definitions
  it satisfies 
  \[
    (V_F, \varphi(V_F)) = V_{\cl_Q(F)} \in \R^{E \dcup \{e\}},
  \]
  and therefore the graph of $\varphi$ is contained in $\trop(Q)$.
  By remark \ref{Q} we can use lemma \ref{geometric} to see that in fact $\trop(Q)$
  is the tropical completion of the graph (as in e.g.\ 
  \cite{AR}*{construction 3.3}), i.e.\ the unique tropical cycle
  containing the graph and with additional facets only in direction
  $V_{\{e\}}$. Therefore $\trop(N) = \trop(Q)^{\cap e}$ is exactly
  the divisor of $\varphi$. 
\end{proof}

Let us now collect the previous results in the following important
corollary.

\begin{corollary}
Let $M,N$ be matroids such that $\trop(N)$ is a codimension $k$ subcycle of $\trop(M)$. Then there are rational functions $\varphi_1,\ldots,\varphi_k$ such that $\varphi_1\cdots\varphi_k\cdot\trop(M)=\trop(N)$.
\end{corollary}

\section{The intersection product on matroid varieties}

Our next aim is to use the results of the previous section 
to find rational functions cutting out the diagonal $\Delta_{\trop(M)}$ in the product $\trop(M)\times\trop(M)$.
In fact, the only thing which is left to do is to observe that both 
$\Delta_{\trop(M)}$ and $\trop(M)\times\trop(M)$ are indeed 
matroid varieties.
We know already from lemma \ref{directsum} that 
$\trop(M)\times\trop(M) = \trop(M \oplus M)$.
Next, we give the necessary definition concerning 
the diagonal $\Delta_{\trop(M)}$. Here, $\Delta_{\trop(M)}$
denotes the push-forward of $\trop(M)$ along the map
$\trop(M) \rightarrow \trop(M) \times \trop(M),$ $x \mapsto (x,x)$.

\begin{definition}
Let $M$ be a matroid on the set $E$. We define
$\diagM$ to be the matroid having the ground set $E \dcup E$ 
and the rank function $\rank_{\diagM}(A\dcup B):=\rank_M(A\cup B)$.
\end{definition}

The criteria of theorem \ref{criterionrank} can be easily checked for the
function $\rank_{\diagM}$, so this really defines a matroid.
It is also easy to see that $\{F\dcup F: F \text{ flat in } M\}$ 
is the set of flats in $\diagM$. Therefore, $\betrag{\mf(\diagM)}=\betrag{\Delta_{\bfan}}$, and we can conclude that the cycles 
\[ 
  \trop(\diagM) = \Delta_{\trop(M)}
\] 
are equal. Now we are ready to state the following main result.

\begin{corollary} \label{diagonal}
  Let $M$ be a matroid of rank $r$. Then there exist piecewise linear
  functions $\varphi_1, \ldots, \varphi_r$ on $\trop(M) \times \trop(M)$
  which cut out the diagonal $\Delta_{\trop(M)}$, i.e.\
  \[
    \Delta_{\trop(M)} = \varphi_1 \cdots \varphi_r \cdot \trop(M) \times \trop(M).
  \]
  In fact, for the containment 
  $\trop(\diagM) \subseteq \trop (M \oplus M)$ the intermediate matroids 
  $M_i$ from corollary \ref{zwischenmatroid} can be computed to have
  rank function
  \[
    \rank_{M_i}(A \dcup B) = 
      \min\{\rank_M(A\cup B) + i,\; \rank_M(A)+\rank_M(B)\},
  \]
  and following proposition \ref{modification} we can choose $\varphi_i$
  to be given by
  \[
    \varphi_i(V_F)=\begin{cases}
      -1, & \text{ if }\rank_M(A)+\rank_M(B)-\rank_M(A\cup B)\geq i \\ 
      0,  & \text{ else}  
    \end{cases},
  \] 
  where $F = A \dcup B$ is a flat of $M \oplus M$.

\end{corollary}

As an immediate consequence of this fact, in complete analogy to 
\cite[definition 9.3]{AR} and \cite[definition 1.16]{lars}, we can now define an intersection product of cycles
in matroid varieties.

\begin{definition}
\label{intersectionproduct}
Let $C$, $D$ be subcycles of $\trop(M)$ of codimension $s$ and $p$. We define the \emph{intersection product} $C\cdot D\in\Zy_{r-s-p}(\trop(M))$ of the cycles $C$ and $D$ in $\trop(M)$ as \[C\cdot D=\pi_{\ast}(\varphi_r\cdots\varphi_1\cdot C\times D),\] where $\pi:\trop(M)\times\trop(M)\rightarrow\trop(M)$ is the projection to the first factor.
\end{definition}

Note that here and in the following, we a priori stick to the 
definition of the functions $\varphi_i$
in the last part of corollary \ref{diagonal}. However, we will see later
that the definition is independent of all choices. 
There is only
one lemma to prove before we can list the basic properties of
the intersection product.

\begin{lemma}
\label{containedindiagonal}
Let $C,D$ be cycles in $\trop(M)$. Then $\varphi_r\cdots\varphi_1\cdot C\times D$ is a 
subcycle of $\diagonal$.
In particular, the definition of $C \cdot D$ does not depend on the chosen projection. 
\end{lemma}

\begin{proof}
We prove by induction over $k$ that 
\[\betrag{\varphi_k\cdots\varphi_1\cdot C\times D}\subseteq\betrag{\varphi_k\cdots\varphi_1\cdot\trop(M)\times\trop(M)}.\] 
for all $k = 1, \ldots, r$, where the case $k=r$ proves the claim.
It is clear that \[ \betrag{\varphi_k\cdot\varphi_{k-1}\cdots\varphi_1\cdot C\times D}\subseteq |{\varphi_k}_{\mid\betrag{\varphi_{k-1}\cdots\varphi_1\cdot C\times D}} |,\] where the right hand side is the locus of non-linearity of the restriction of $\varphi_k$ to the support of $\varphi_{k-1}\cdots\varphi_1\cdot C\times D$. By the induction hypothesis, the right hand side is contained in \[|{\varphi_k}_{\mid\betrag{\varphi_{k-1}\cdots\varphi_1\cdot \trop(M)\times\trop(M)}}|.\] Since $\varphi_{k-1}\cdots\varphi_1\cdot \trop(M)\times\trop(M)$ is a matroid variety, and hence locally irreducible, it follows by \cite[1.2.31]{disshannes} that \[|{\varphi_k}_{\mid\betrag{\varphi_{k-1}\cdots\varphi_1\cdot \trop(M)\times\trop(M)}}|=\betrag{\varphi_k\cdots\varphi_1\cdot\trop(M)\times\trop(M)}.\] 
\end{proof}

\begin{theorem} \label{properties}
For all subcycles $C, D, E$ of $\trop(M)$, the following properties hold:
\begin{enumerate}
\item $\betrag{C\cdot D}\subseteq \betrag{C}\cap\betrag{D}$.
\item If $C$ and $D$ are fans, then $C\cdot D$ is a fan, too.
\item $(\varphi\cdot C)\cdot D=\varphi\cdot (C\cdot D)$ for any Cartier divisor $\varphi$ on $C$.
\item $C\cdot\trop(M)=C$.
\item $C\cdot D=D\cdot C$.
\item If $C=\psi_1\cdots\psi_s\cdot\trop(M)$, then $C\cdot D=\psi_1\cdots\psi_s\cdot D$.
\item $(C\cdot D)\cdot E=C\cdot (D\cdot E)$.
\item $(C+D)\cdot E=C\cdot E+D\cdot E$.
\end{enumerate}
\end{theorem}

\begin{proof}
(1) follows directly from lemma \ref{containedindiagonal}. 
Everything else except (4) follows either directly or can be deduced in
exactly the same way as in the $\R^r$-case (cf.\ 
\cite[1.5.2, 1.5.5, 1.5.6, 1.5.9]{disshannes} or \cite[section 9]{AR}). \\
It remains to prove (4).
By (8) it suffices to prove (4) for irreducible cycles $C$. We know by (1) that $\betrag{C\cdot\trop(M)}\subseteq\betrag{C}$; hence the irreducibility of $C$ implies that $C\cdot \trop(M)=\lambda_C\cdot C$  for some $\lambda_C\in\Z$. We first note that the factors $\lambda_P$ are the same for every point $P$ in $\trop(M)$: For any point $P$, the recession fan of $P\times \trop(M)$ is $\{0\}\times\trop(M)$; thus we know by \cite[proposition 2.2.2]{disslars} that \[\lambda_P=\deg(\varphi_r\cdots\varphi_1\cdot P\times \trop(M))= \deg(\varphi_r\cdots\varphi_1\cdot\{0\}\times \trop(M))=\lambda_{\{0\}}.\]  
Now, as it was done in the proof of \cite[proposition 1.4.15]{disshannes}, we choose rational functions $\psi_1\ldots,\psi_{\dim(C)}$ such that $\psi_1\cdots\psi_{\dim(C)}\cdot C\neq 0$. Then (3) implies that
\begin{eqnarray*}
\lambda_C\cdot(\psi_1\cdots\psi_{\dim(C)}\cdot C)&=&\psi_1\cdots\psi_{\dim(C)}\cdot (C\cdot\trop(M))\\&=&(\psi_1\cdots\psi_{\dim(C)}\cdot C)\cdot\trop(M)\\&=&\lambda_{\{0\}}\cdot(\psi_1\cdots\psi_{\dim(C)}\cdot C).
\end{eqnarray*}
Hence $\lambda_C=\lambda_{\{0\}}$ for all cycles $C$. As $\lambda_{\trop(M)}=1$, it follows that $C\cdot\trop(M) =C$ for every $C$.
\end{proof}

\begin{remark}
\label{independent}
It follows from theorem \ref{properties} (6) that our intersection product is independent of the choice of rational functions describing the diagonal $\diagonal$, 
as each intersection product can be calculated as 
\[
  C \cdot D = \pi_*(\diagonal \cdot C \times D),
\]
where the right hand side is an intersection product of cycles on 
$\trop(M \oplus M)$.
This is not only satisfactory, but also what we need to prove the 
following lemmas.
\end{remark}

\begin{lemma} \label{automorphism}
  Let $\alpha : \trop(M) \rightarrow \trop(M')$ be a tropical isomorphism of matroid 
  varieties and let $C$ and $D$ be two arbitrary cycles in $\trop(M)$.
  Then the followings holds.
  \[
    \alpha_*(C \cdot D) = \alpha_*(C) \cdot \alpha_*(D)
  \]
\end{lemma}

\begin{proof}
  If $C=\psi_1\cdots\psi_s\cdot\trop(M)$ is cut out by rational functions,
  the claim follows from theorem \ref{properties} (6) and the projection formula.
  We apply this to $\beta := \alpha \times \alpha$ (the corresponding isomorphism between
  $\trop(M) \times \trop(M)$ and $\trop(M') \times \trop(M')$) and the cycles $\diagonal$
  and $C \times D$. 
  By the previous remark, this suffices to prove the claim.
\end{proof}

\begin{lemma}
\label{crossproduct}
Let $A_1,B_1$ be cycles in $\trop(M_1)$ and let $A_2, B_2$ be cycles in $\trop(M_2)$. Then 
\[(A_1\times A_2)\cdot (B_1\times B_2)=(A_1\cdot B_1)\times (A_2\cdot B_2).\]
\end{lemma}

\begin{proof}
If $A_1$ and $A_2$ are cut out by rational functions, the claim follows from theorem \ref{properties} (6). The general statement follows from remark \ref{independent} using
the fact that (after permuting the coordinates)
$\Delta_{\trop(M_1)} \times \Delta_{\trop(M_2)} =\Delta_{\trop(M_1)\times\trop(M_2)}$. 
\end{proof}

\begin{remark}
  Let $\trop(N), \trop(N')$ be two matroid varieties contained
  in third matroid variety $\trop(M)$. 
  So far, we were not able to find an easy matroid-theoretic
  description of the intersection product $\trop(N) \cdot \trop(N')$. 
  In general, the product is not just a matroid variety again.
  The easiest example where we at least get negative weights is
  the self-intersection of the straight line contained in the plane 
  $\max\{0,x,y,z\} \cdot \R^3$
  (cf. \cite{AR}*{example 3.10}). 
  In our setup, this is given by
  $\trop(N) \subseteq \trop(M)$, where $M$ is the uniform matroid
  of rank $3$ on $4$ elements (cf. example \ref{lnk}) and $N$
  is the matroid with lattice of flats $\emptyset, \{1,2\},
  \{3,4\}, \{1,2,3,4\}$. It is easy to check that the self-intersection
  of $\trop(N)$ in $\trop(M)$ is $\R\cdot(1,1,1,1)$ with weight $-1$. 
  
  To find a general description of $\trop(N) \cdot \trop(N')$, one should
  probably compactify the problem in $\T\PP^n$, but this causes other 
  difficulties. The only case which is understood so far is $\trop(M) = \R^n$.
  In this case, we form the \emph{matroid intersection} $N \wedge N'$
  (cf.\ \cite{White}*{section 7.6}).
  The bases of $N \wedge N'$ are the minimal sets in
  \[
    \{B \cap B' : B \text{ basis of } N, B' \text{ basis of } N'\}.
  \]
  If $r,s$ are the ranks of $N,N'$, then the rank of $N \wedge N'$
  is greater or equal to $n - r - s$ (where equality is attained
  if and only if there exist bases $B,B'$ of $N,N'$ satisfying
  $B \cup B' = [n]$). Then we get
  \[
    \trop(N) \cdot \trop(N') =
      \begin{cases}
        \trop(N \wedge N'), & \text{if the rank of $N \wedge N'$ 
                                   is $n-r-s$}, \\
        \emptyset,          & \text{otherwise}.
      \end{cases}
  \]
  This follows essentially from extending the arguments in the proof
  of \cite{Speyer}*{proposition 3.1} to the case where the coordinates
  of the Pl\"ucker vector are allowed to be infinite.
\end{remark}

\section{Dividing out the lineality space} \label{dividingout}

So far, we defined an intersection product on $\trop(M)$ which is a ``tropical cone''
in the sense that it contains the lineality space $L = \R\cdot (1,\ldots, 1)$. 
But in most applications, one is really interested in $\trop(M) / L$. We will 
now discuss how the intersection product of $\trop(M)$ descends to $\trop(M)/L$.
First, let us fix some terminology.

Let $\curlyx$ be a polyhedral complex in a vector space $V$. 
For a cell $\tau \in \curlyx$ we denote by $V_\tau$ 
the linear subspace spanned by (differences of vectors in) $\tau$. 
The intersection of all these subspaces $L := \cap_{\tau \in \curlyx} V_\tau$
is called the \emph{lineality space} of $\curlyx$. 
If $\curlyx$ is a fan, $L$ is just the unique inclusion minimal cone of 
$\curlyx$. We define the polyhedral complex $\curlyx/L$ in $V/L$ by 
$\curlyx/L:=\{\qqq(\tau)|\tau \in \curlyx\}$, where $\qqq:V\rightarrow V/L$ is 
the quotient map. If $\curlyx$ is weighted, $\qqq(\sigma)$ inherits the weight 
from $\sigma$. 

Let $X$ be a tropical cycle in $V$. A subspace $L \subseteq V$ is called a
\emph{lineality space} of $X$ if there is a polyhedral structure $\curlyx$ of $X$ 
whose lineality space is $L$. In this case, we denote by $X/L$ the tropical
cycle in $V/L$ represented by $\curlyx /L$.

Let $C$ be a cycle in $X/L$ and let $\curlyc$ be a polyhedral structure of $C$. 
We define the polyhedral complex $\qqq^{-1}(\curlyc)$ in $X$ to be
the collection of cells $\{\qqq^{-1}(\sigma) | \sigma\in\curlyc\}$ (with
weights inherited from $\curlyc$). 
Furthermore, we define $\qqq^{-1}(C)$ to be the tropical 
cycle associated to $\qqq^{-1}(\curlyc)$. By definition, $L$ is a lineality 
space of $\qqq^{-1}(C)$.

Now, the only thing we need in order to define an intersection product
on $\trop(M)/L$ is the following lemma.

\begin{lemma} \label{linspace}
  Let $C,D$ be two cycles in a matroid variety $\trop(M)$ and let us assume
  that $L$ is a lineality space of each. Then $L$ is also a lineality
  space of $C \cdot D$ (if non-zero).
\end{lemma}

\begin{proof}
  For all vectors $v \in L$ we can define the translation automorphism
  $\alpha_v : \trop(M) \rightarrow \trop(M)$ which sends $x$ to $x+v$.
  For a subcycle of $\trop(M)$, having $L$ as lineality space is equivalent
  to being invariant under all translations $\alpha_v, v \in L$.
  Now we use lemma \ref{automorphism} to see that this property
  is passed from $C$ and $D$ to $C \cdot D$.
\end{proof}

\begin{definition} \label{defintprodlinspace}
  Let $\trop(M)$ be a matroid variety with lineality space $L$, and let
  $C,D$ be two tropical cycles in $\trop(M) / L$. We define the 
  \emph{intersection product} $C \cdot D$ of $C$ and $D$ in $\trop(M) / L$ by 
  \[
    C \cdot D := (\qqq^{-1}(C) \cdot \qqq^{-1}(D))/L,
  \]
  where on the right hand side we use the previously defined intersection
  product on $\trop(M)$ (cf.\ definition \ref{intersectionproduct}).
  In words, we first take preimages of $C$ and $D$ in $\trop(M)$ and intersect 
  them. By lemma \ref{linspace}, the result has lineality space $L$ which we 
  divide out again. 
\end{definition}

\begin{remark} \label{cart products}
  This definition also works for cartesian
  products $\trop(M)/L \times \trop(M')/L'$ as they are equal to
  $\trop(M \oplus M')/L \times L'$.
\end{remark}

\begin{proposition}
\label{diagonallinspace}
Let $C,D$ be cycles in $ \trop(M)/L$. Then $\Delta_{\trop(M)/L}\cdot (C\times D)=\Delta_{C\cdot D}$. In particular, we have
\[C\cdot D=\pi_{\ast}(\Delta_{\trop(M)/L}\cdot C\times D),\] 
where $\pi:\trop(M)/L\times\trop(M)/L\rightarrow\trop(M)/L$ is the projection to the first factor. Note that this is how we defined our intersection product on matroid varieties (cf.\ remark \ref{independent}).
\end{proposition}

To prove this we use the following lemmata:

\begin{lemma}
  Let $X$ be a tropical cycle with polyhedral structure $\curlyx$ 
  whose lineality space is $L$. Let $\varphi$ be a function which
  is affine linear on the cells of $\curlyx$ and let $C$ be a cycle
  in $X$ (not necessarily with lineality space $L$). Then the equation
  \[
    \varphi \cdot \qqq^{-1}\qqq_*(C) = \qqq^{-1}\qqq_*(\varphi \cdot C)
  \]
  holds, where $\qqq : X \rightarrow X/L$ is the quotient map.
\end{lemma}

\begin{proof}
  First note that by adding a globally affine linear function to $\varphi$, we can
  assume $\varphi = \qqq^*\widetilde{\varphi}$ for a suitable function 
  $\widetilde{\varphi}$ on $X/L$. In this case, it is obvious from the definitions
  and projection formula that both sides equal 
  $\qqq^{-1}(\widetilde{\varphi} \cdot \qqq_*(C))$.
\end{proof}

\begin{lemma}
\label{qminus1}
Let $L$ be a lineality space of a matroid variety $\trop(M)$, and $\qqq:\trop(M)\rightarrow\trop(M)/L$ the corresponding quotient map. Let $C,D$ be cycles in $\trop(M)$ such that 
$L$ is a lineality space of $D$. Then $\qqq^{-1}\qqq_*(C)\cdot D=\qqq^{-1}\qqq_*(C\cdot D)$.
\end{lemma}

\begin{proof}
First, we split $M$ into its connected components $M=\bigoplus_{i}M_i$
and pull back the functions that cut out the diagonal of 
$\trop(M_i) \times \trop(M_i)$ to $\trop(M) \times \trop(M)$. 
With the help of lemma \ref{connectedcomponents} this gives us functions
on $\trop(M) \times \trop(M)$ which cut out the diagonal and are affine linear
on a polyhedral structure of $\trop(M) \times \trop(M)$ with lineality space
$\Delta_L$.

Second, set
$\jj: \trop(M) \times \trop(M) \rightarrow (\trop(M) \times \trop(M))/\Delta_L$.
Then we have $\qqq^{-1}\qqq_*(C) \times D = \jj^{-1}\jj_*(C \times D)$.
Thus we are in the situation of the previous lemma, and intersecting with the diagonal
gives $\jj^{-1}\jj_*(\diagonal \cdot C \times D)$. After projecting, this
is $\qqq^{-1}\qqq_*(C\cdot D)$ and we are done.
\end{proof}

\begin{proof}[Proof of proposition \ref{diagonallinspace}]
Let $\qqq:\trop(M)\rightarrow\trop(M)/L$ be the quotient map. For any cycle $A$ having lineality space $L$ the following equality holds:
\begin{equation}\label{E}
(\qqq\times\qqq)^{-1}\Delta_{A/L}=(\id\times\qqq)^{-1}(\id\times \qqq)_{\ast}\Delta_A.
\end{equation}
The set-theoretic equality is clear; the equality of the cycles follows from the fact that all involved weights are inherited by the weights of $A$. 
By definition of our intersection products and equation \eqref{E} we have
\begin{eqnarray*}
\Delta_{\trop(M)/L}\cdot (C\times D) &=& ((\qqq\times\qqq)^{-1}\Delta_{\trop(M)/L}\cdot \qqq^{-1}C\times \qqq^{-1}D)/L\times L \\
&=&  (((\id\times\qqq)^{-1}(\id\times\qqq)_{\ast}\Delta_{\trop(M)})\cdot \qqq^{-1}C\times \qqq^{-1}D)/L\times L,
\end{eqnarray*}
as well as
\begin{eqnarray*}
\Delta_{C\cdot D} &=& ((\qqq\times\qqq)^{-1}(\Delta_{(\qqq^{-1}C\cdot\qqq^{-1}D)/L}))/L\times L \\
&=& ((\id\times\qqq)^{-1}(\id\times\qqq)_{\ast}(\Delta_{\qqq^{-1}C\cdot\qqq^{-1}D}))/L\times L \\
&=& ((\id\times\qqq)^{-1}(\id\times\qqq)_{\ast}(\Delta_{\trop(M)}\cdot \qqq^{-1}C\times \qqq^{-1}D))/L\times L. 
\end{eqnarray*}
Now the claim follows from lemma \ref{qminus1}.
\end{proof}

\begin{remark}\label{abc}
  Let $\trop(M)/L$ be a quotient of a matroid variety and assume we can cut out
  the diagonal $\Delta_{\trop(M)/L}$ in $\trop(M)/L \times \trop(M)/L$ with a 
  collection of rational functions. Then the intersection product defined by this
  collection coincides with the one defined in \ref{defintprodlinspace}.
  This follows from proposition \ref{diagonallinspace} together with 
  property (6) of theorem \ref{properties}.
\end{remark}

\begin{remark} \label{automorphism linspace}
  Lemma \ref{automorphism} also holds if we replace 
  $\trop(M)$ by $\trop(M)/L$, i.e.\ we have 
  \[
    \alpha_*(C \cdot D) = \alpha_*(C) \cdot \alpha_*(D)
  \]
  for any isomorphism of $\alpha : \trop(M)/L \rightarrow \trop(M')/L'$. 
  We first use remark \ref{cart products} and write 
  $\trop(M)/L = (\trop(M) \times L')/(L \times L')$ resp.\ 
  $\trop(M')/L' = (L \times \trop(M'))/(L \times L')$. 
  In other words, we can assume that $\trop(M)$ and $\trop(M')$ lie in the
  same ambient vector space and that $L=L'$. 
  In this situation we can
  lift $\alpha$ to an isomorphism 
  $\widetilde{\alpha} : \trop(M) \rightarrow \trop(M')$
  with $q \circ \widetilde{\alpha} = \alpha \circ q$ and use
  lemma \ref{automorphism}. 
\end{remark}

\section{Smooth varieties and locality}

For the sake of completeness, in the following we give a (preliminary) definition
of smooth tropical varieties (whose local models
are matroid varieties modulo lineality spaces) 
and extend the intersection product to those.
For the latter we have to show that the intersection product can be computed ``locally''.

\begin{definition}
  A \emph{smooth tropical variety} is a  
  topological space $X$ together with an 
    open cover $\{U_i\}$ and homeomorphisms
  \[
    \phi_i : U_i \rightarrow W_i \subseteq |\trop(M)/L| \subseteq \R^{n}/L
  \]
  such that
  \begin{itemize} 
  \item
    each $W_i$ is an (euclidean) open subset of $|\trop(M)/L|$ for a suitable
    matroid $M$ with (suitable) lineality space $L$; 
  \item
    for each pair $i,j$, the transition map
    \[
      \phi_j \circ \phi_i^{-1} :
        \phi_i(U_i \cap U_j) \rightarrow \phi_j(U_i \cap U_j)
    \]
    is the restriction of an affine $\Z$-linear map $\Phi_{i,j}$, i.e.\
    the composition of a translation by a real vector and a $\Z$-linear map.
  \end{itemize}
\end{definition}

Let us stress again that this is only a provisional definition which is appropriate
for the purposes of this paper. 
In particular, our definition does not allow any boundary points (i.e.\ 
points of ``positive sedentarity''). However, we chose not to 
reflect this in a more complicated name. 

Note also that a tropical cycle $X$
in $\R^n$ is a smooth variety if and only if for all points $p$ in $X$
the star $\Star_X(p)$ is isomorphic to $\trop(M)/L$, the quotient of
a suitable matroid variety.

Let $W \subseteq |\trop(M)/L| \subseteq \R^{n}/L$ be a set as in the previous
definition. We can define polyhedral complexes and tropical cycles in $W$
exactly as in $\trop(M)/L$ --- by just defining a polyhedron in $W$ to be
the (non-empty) 
intersection of a polyhedron in $|\trop(M)/L| \subseteq \R^{n}/L$ with $W$.
A set $C \subseteq W$ (or, more generally, a topological space)
is called a \emph{weighted set} if each point from a dense open subset of $C$ is equipped with a non-zero integer weight which is locally constant
(in the dense open subset). Two such weighted sets $C,C'$ are said to \emph{agree} and
are thus identified if the sets are equal and the weight functions agree (where
both defined). Note that each tropical cycle $D$ in $W$ can be regarded
as a weighted set
by inheriting the weight of each facet to its interior points.

\begin{definition}
  Let $X$ be a smooth tropical variety. A \emph{tropical subcycle} of $X$ 
  is defined to be
  a weighted set $C$ such that for all $i$ the induced weighted set
  $\phi_i(C \cap U_i)$ agrees with a tropical cycle in $W_i$.
\end{definition}

Of course, each smooth variety $X$ contains the \emph{fundamental cycle} $X$ itself
with constant weight $1$ for all points.

If $C,D$ 
are two tropical cycles in $X$, we want to define their 
intersection product in two steps: We first intersect $C$ and $D$
locally on each $U_i$ (via $\phi_i$) and then glue together the
local results. To make this approach work, it remains to be checked that
intersection products on $\trop(M)/L$ can
be computed locally, as in the $\R^r$-case (cf.\ 
\cite{disshannes}*{proposition 1.5.8}).

Let $X$ be a tropical cycle in a vector space $V$ and let $p$ be a point
in $X$. Recall that we defined $\Star_X(p)$ to be 
the fan cycle containing all vectors $v$ such that
$p + \epsilon v \in |X|$ for sufficiently small $\epsilon$.
Let $\varphi$ be a rational function on $X$. Then $\varphi$ induces
a function $\varphi^p$ on $\Star_X(p)$. Namely, we first restrict
$\varphi$ to a small neighbourhood of $p$ and then extend it by linearity 
to $\Star_X(p)$
(one might also normalise to $\varphi^p(0) = 0$). From the locality of
$\varphi \cdot X$ it follows that
\[
  \Star_{\varphi \cdot X}(p) = \varphi^p \cdot \Star_X(p)
\]
(cf.\ \cite[proposition 1.2.12]{disshannes}). 

In particular, if the functions $\varphi_1, \ldots, \varphi_r$ cut out
the diagonal of $X$ in $X \times X$, then the functions
$\varphi_1^{(p,p)}, \ldots, \varphi_r^{(p,p)}$ cut out the diagonal of
$\Star_X(p)$ in $\Star_{X \times X}(p,p) = \Star_X(p) \times \Star_X(p)$.
Using these collections of functions to define intersection products on
$X$ resp.\ $\Star_X(p)$, we obviously get the equality
\[
  \Star_{C \cdot D}(p) = \Star_C(p) \cdot \Star_D(p),
\]
where the left hand side (resp.\ right hand side) contains a product on
X (resp.\ $\Star_X(p)$).

Note again that the intersection product on $\trop(M)$ is independent of
the chosen functions. Moreover, for each cycle $C$ in $\trop(M)/L$ and 
$p \in |\qqq^{-1}(C)|$ with $\qqq(p) = p'$ 
we have $\qqq^{-1}\Star_C(p') = \Star_{\qqq^{-1}C}(p)$.
This leads to the following statement.

\begin{corollary}
Let $C,D$ be subcycles of $\trop(M)/L$ and $p$ a point of $\trop(M)/L$. 
Then the following equality holds: 
\[
  \Star_{C \cdot D}(p) = \Star_C(p) \cdot \Star_D(p)
\]  
\end{corollary}

Let us now put things together. First, we have a well-defined 
intersection product on open sets 
$W \subseteq |\trop(M)/L| \subseteq \R^{n}/L$.
Namely, for two cycles $C,D$ in $W$, $C \cdot D$ is the subcycle
of $W$ which satisfies 
$\Star_{C \cdot D}(p) = \Star_C(p) \cdot \Star_D(p)$, where
the latter part of the equation is an honest intersection 
product of the two subcycles $\Star_C(p), \Star_D(p)$ on
$\trop(M_p)/L$. 
This does not depend on the choice of
$M,L$ by remark \ref{automorphism linspace}.

When $C,D$ are two subcycles of a smooth tropical variety $X$,
then on each $U_i$ we can compute 
$E_i := (C \cap U_i) \cdot (D \cap U_i)$.
Using locality again, we see that on each overlap $U_i \cap U_j$
the weighted sets $E_i$ and $E_j$ agree. More precisely, this
follows from the fact that for each point $p \in U_i \cap U_j$,
the maps $\Phi_{i,j}$ resp.\ $\Phi_{j,i}$ provide isomorphisms
between the stars of $W_i, \phi_i(C), \phi_i(D)$ at $\phi_i(p)$
on the one hand and the stars of 
$W_j, \phi_j(C), \phi_j(D)$ at $\phi_j(p)$ on the other hand.
Therefore, also the stars of the local intersections 
$\phi_i(E_i)$ and $\phi_j(E_j)$ at $p$ are isomorphic
(cf.\ remark \ref{automorphism linspace}), which 
proves that $E_i$ and $E_j$ agree locally.
We collect all this in the following theorem.

\begin{defthm} \label{properties1}
  Let $X$ be a smooth tropical variety and let $C$ and $D$ be subcycles of $X$.
  Then the \emph{intersection product of $C$ and $D$ on $X$}, denoted by $C \cdot D$,
  is the unique subcycle of $X$ such that
  \[
    (C\cdot D) \cap U_i = (C \cap U_i) \cdot (D \cap U_i)
  \]
  holds for any $U_i$ of the open cover. Moreover, this intersection product satisfies
  the following properties.
\begin{enumerate}
\item $\codim(C \cdot D) = \codim C + \codim D$ (if $C\cdot D \neq 0$).
\item $\betrag{C\cdot D}\subseteq \betrag{C}\cap\betrag{D}$.
\item $(\varphi\cdot C)\cdot D=\varphi\cdot (C\cdot D)$ for any Cartier divisor $\varphi$ on $C$.
\item $C\cdot X=C$.
\item $C\cdot D=D\cdot C$.
\item If $C=\Psi_1\cdots\Psi_s\cdot X$, then $C\cdot D=\Psi_1\cdots\Psi_s\cdot D$.
\item $(C\cdot D)\cdot E=C\cdot (D\cdot E)$.
\item $(C+D)\cdot E=C\cdot E+D\cdot E$.
\item $(A_1\times A_2)\cdot (B_1\times B_2)=(A_1\cdot B_1)\times (A_2\cdot B_2)$ 
if $A_1,B_1$ and $A_2, B_2$ are subcycles of the two smooth varieties $X_1$ and $X_2$ respectively.
\end{enumerate}
\end{defthm}

\begin{proof}
  We already discussed that our definition is well-defined.
  It remains to show the list of properties. 
  In the case $X = \trop(M)$, all the properties have already been proven.
  The next step is $X = \trop(M)/L$, to which the properties immediately
  generalise. For the general case, note that all properties can be 
  verified locally. Therefore, by our local definition of the general
  intersection product, all properties also hold in the general case.
\end{proof}

\begin{remark}
  Let $X,Y$ be two smooth varieties. A tropical morphism $f : X \rightarrow Y$
  is a continuous map such that for all $i,j$ the map 
  $\phi_j^Y \circ f \circ (\phi_i^X)^{-1}$ on the charts is induced by an affine $\Z$-linear map
  of the ambient vector spaces. We call $f$ an isomorphism if there is an 
  inverse tropical morphism $g : Y \rightarrow X$.
  We can extend remark \ref{automorphism linspace} to this case, i.e.\
  if $f$ is an isomorphism and $C$ and $D$ are two subcycles of $X$, then 
  \[
    f_*(C \cdot D) = f_*(C) \cdot f_*(D).
  \]
  Moreover, we can extend proposition \ref{diagonallinspace} and
  check locally that 
  \[
    C \cdot D = \pi_*(\Delta_X \cdot C \times D)
  \]
  holds for all smooth varieties $X$.
\end{remark}

\section{Examples}

In this section we discuss a few examples. The first example compares our 
new definitions to the previously known cases of $\R^r$ and $L^n_k$.
The following examples are devoted to the moduli spaces of tropical rational curves.

\begin{example}
\label{lnk}
Let $M=U_{k+1,n+1}$ be the uniform matroid of rank $k+1$ on the set $N:=\{1,\ldots,n+1\}$ (i.e.\
each $k+1$-subset of $N$ is a basis). 
Then $L:=\R\cdot (1,\ldots ,1)$ is the lineality space of $\trop(M)$, and $\trop(M)/L$ is isomorphic to $L^n_{k}=\max\{x_1,\ldots,x_n,0\}^{n-k}\cdot\R^n$. Thus we have reproved the result of \cite{lars} that the cycles $L^n_{k}$ admit an intersection product of cycles. Note that both intersection products agree by remark \ref{abc}.
\end{example}

\begin{example}
\label{mn}
The complete (undirected) graph $K_{n-1}$ with $n-1$ vertices defines a matroid on the set of edges $\left\{1,\ldots,\binom{n-1}{2}\right\}$ whose independent sets are the trees in $K_{n-1}$. 
It was shown in \cite[chapter 4]{ardila} that $\trop(K_{n-1})$ parameterises 
so-called \emph{equidistant $(n-1)$-trees} (i.e.\ rooted trees with $n-1$ 
labelled leaves and lengths
on each edge such that the distance from the root to any leaf is the same).
As a variation of this, we construct a bijection of  $\trop(K_{n-1})/L$ (with $L:=\R\cdot (1,\ldots ,1)$) and $\mn$, the space of $n$-marked abstract rational tropical curves
(i.e.\ metric trees with bounded internal edges and $n$ unbounded labelled leaves; 
see \cite[chapter 3]{GKM} for the construction of $\mn$).
Our bijection is analogous to the one in \cite{ardila} except for a global scalar factor.
More important, we show that this map is actually a tropical isomorphism of the two
fans, i.e.\ it is induced by a $\Z$-linear transformation of the ambient spaces.
Hence $\mn$ can also be equipped with an intersection product of cycles. 

Note that $\trop(K_{n-1})/L$ lives in $\R^{\binom{n-1}{2}}/L$, whereas
$\mn$ lives in $\R^{\binom{n}{2}}/\Ima(\phi_n)$. 
Here, $\phi_n:\R^n\rightarrow \R^{\binom{n}{2}}$ is the linear map defined by $(a_1,\ldots,a_n)\mapsto ((a_i+a_j))_{i,j}$. We define the linear map $f$ by 
\begin{eqnarray*}
f:  \R^{\binom{n-1}{2}}/L &\rightarrow & \R^{\binom{n}{2}}/\Ima(\phi_n) \\  (a_{i,j})_{i,j} &\mapsto & (b_{i,j})_{i,j} \ , \ \ \ \text{ with } b_{i,j}=\begin{cases} 0, & \text{ if } n\in\{i,j\} \\ 2\cdot a_{i,j}, & \text{ else}\end{cases}.
\end{eqnarray*} 
It is easy to see that $f$ is well-defined and injective. Since its domain and target space have the same dimension, it follows that $f$ is a linear transformation. 

Let $F$ be a flat of the matroid corresponding to $K_{n-1}$. 
Then $F$ is a vertex-disjoint union of complete subgraphs $S_1,\ldots,S_p$ of $K_{n-1}$, and \[f(V_F)=(b_{i,j})_{i,j},  \text{ with } b_{i,j}=\begin{cases}-2, & \text{ if } \{i,j\}\subseteq V(S_t) \text{ for some } t \\ 0, & \text{ else }\end{cases},\] where the $V(S_t)$ denote the sets of vertices of the complete subgraphs $S_t$. 
We define a vector $a \in \R^n$ by setting $a_i=1$ if $i\in V(S_t)$ for some $t$, and $a_i=0$ otherwise. Then   
\[(f(V_F)+\phi_n(a))_{i,j}=\begin{cases}0, & \text{ if } \{i,j\}\subseteq V(S_t) \text{ for some } t, \text{ or } i,j\not\in V(S_t) \text{ for all } t \\ 1, & \text{ if } i\in V(S_t) \text{ for some } t, \text{ and } j\not\in V(S_s) \text{ for all } s\\ 2, & \text{ if there are $s\not = t$ with } i\in V(S_s),\  j\in V(S_t)  \end{cases}.\] 
The metric graph with $n$ leaves associated to this vector, denoted by $M_F$, 
is depicted in the following picture.
\begin{center}
\begin{picture}(0,0)%
\includegraphics{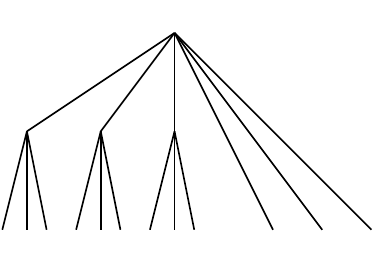}%
\end{picture}%
\setlength{\unitlength}{2072sp}%
\begingroup\makeatletter\ifx\SetFigFont\undefined%
\gdef\SetFigFont#1#2#3#4#5{%
  \reset@font\fontsize{#1}{#2pt}%
  \fontfamily{#3}\fontseries{#4}\fontshape{#5}%
  \selectfont}%
\fi\endgroup%
\begin{picture}(3419,2425)(654,-1754)
\put(2746,-1681){\makebox(0,0)[lb]{\smash{{\SetFigFont{8}{7.2}{\rmdefault}{\mddefault}{\updefault}{\color[rgb]{0,0,0}$\{1,\ldots,n\}\setminus \cup_{i} S_i$}%
}}}}
\put(756,-1681){\makebox(0,0)[lb]{\smash{{\SetFigFont{8}{7.2}{\rmdefault}{\mddefault}{\updefault}{\color[rgb]{0,0,0}$S_1$}%
}}}}
\put(1431,-1681){\makebox(0,0)[lb]{\smash{{\SetFigFont{8}{7.2}{\rmdefault}{\mddefault}{\updefault}{\color[rgb]{0,0,0}$S_2$}%
}}}}
\put(2151,-1681){\makebox(0,0)[lb]{\smash{{\SetFigFont{8}{7.2}{\rmdefault}{\mddefault}{\updefault}{\color[rgb]{0,0,0}$S_p$}%
}}}}
\put(1830,-1681){\makebox(0,0)[lb]{\smash{{\SetFigFont{8}{7.2}{\rmdefault}{\mddefault}{\updefault}{\color[rgb]{0,0,0}...}%
}}}}
\put(1200,-200){\makebox(0,0)[lb]{\smash{{\SetFigFont{8}{7.2}{\rmdefault}{\mddefault}{\updefault}{\color[rgb]{0,0,0}$1$}%
}}}}
\put(1650,-200){\makebox(0,0)[lb]{\smash{{\SetFigFont{8}{7.2}{\rmdefault}{\mddefault}{\updefault}{\color[rgb]{0,0,0}$1$}%
}}}}
\put(2100,-200){\makebox(0,0)[lb]{\smash{{\SetFigFont{8}{7.2}{\rmdefault}{\mddefault}{\updefault}{\color[rgb]{0,0,0}$1$}%
}}}}
\end{picture}%
\\
\end{center}
We see from this description that $f$ also gives an isomorphism of the underlying
lattices $\Z^{\binom{n-1}{2}}/L \rightarrow \Lambda_n$. 
Here $\Lambda_n$ is the lattice generated 
by metric graphs with only one internal edge of length $1$ (see \cite[construction 3.6]{GKM}). 
$\Z^{\binom{n-1}{2}}/L$ is mapped to $\Lambda_n$ as $\Z^{\binom{n-1}{2}}/L$ is spanned by 
the $V_F$ and $M_F$ is contained in $\Lambda_n$.
Moreover, let $M_{I|J}$ be a generator of $\Lambda_n$ corresponding to the graph whose single internal edge splits the leaves into the partition $I\dcup J=\{1,\ldots,n\}$ with $n\in J$. 
Then we have $M_{I|J} = M_F$, where $F$ is the flat associated to the complete subgraph with vertex set $I$; hence $M_{I|J}$ lies in the image of $\Z^{\binom{n-1}{2}}/L$.

It remains to check that $f$ can be restricted to a bijection
$\trop(K_{n-1})/L \rightarrow \mn$. Sticking to the notation for the generators of$\Lambda_n$, 
we already saw
that $M_F = M_{S_1|S_1^c} + \ldots + M_{S_p|S_p^c}$.
For a chain of flats $\F$, it follows that all the appearing partitions $I\dcup J=\{1,\ldots,n\}$ satisfy the following property: 
For each pair of partitions one part of the first partition is contained in one of the
parts of the second partition. This is what is needed to ensure that each positive
sum of such vectors $M_{I|J}$ still corresponds to a metric graph. Therefore,
the image of $\trop(K_{n-1})/L$ is contained in $\mn$. As $\mn$ is irreducible,
we actually have equality.

Therefore, $f$ induces a tropical isomorphism between $\trop(K_{n-1})/L$ and $\mn$
and thus $\mn$ inherits the intersection product of cycles from $\trop(K_{n-1})/L$.
Note that this intersection product on $\mn$ is independent of the chosen isomorphism by
remark \ref{automorphism linspace}.
\end{example}

\begin{example}
The moduli space $\mnlab(\Delta, \R^r)$ parameterises $n$-marked rational parameterised tropical curves of degree $\Delta$ in $\R^r$ (cf.\ \cite[definition 4.1]{GKM}). It was shown in \cite[proposition 4.7]{GKM} that $\mnlab(\Delta, \R^r)$ can be identified with $\mndelta$.
This identification together with the previous examples shows that we have an isomorphism 
\[
  \mnlab(\Delta, \R^r) \cong 
    \trop(K_{n+\betrag{\Delta}-1} \oplus U_{r+1,r+1})/ L \times L',
\]
where $K_{n+\betrag{\Delta}-1}$ is the complete graph matroid, $U_{r+1,r+1}$ is the 
uniform matroid of rank $r+1$ on $r+1$ elements and $L$ resp.\ $L'$ are their
respectively ``natural'' one-dimensional lineality spaces. This implies that there is an intersection product of cycles on $\mnlab(\Delta, \R^r)$ having the properties listed in theorem \ref{properties1}.
\end{example}

\section{Pull-back of cycles}

When dividing out a lineality space $q : X \rightarrow X/L$, we defined by
$q^{-1}(C)$ a very natural preimage for every cycle $C \in X/L$. 
Moreover,
when we consider a modification $\pi : \widetilde{X} \rightarrow X$ (with $X,\widetilde{X}$ smooth) given by 
a function $\varphi$ on $X$, then for each $C \in X$ there is also a natural 
lift $\widetilde{C}$ of $C$ to $\widetilde{X}$. Namely, we can restrict $\varphi$ 
to $C$ and define $\widetilde{C}$ to be the modification of $C$ by $\varphi|_C$.

In the following, we will see that both cases are examples of a more general
construction. 
This generalisation is useful when dealing with a chain of several
modifications and when showing that our intersection product agrees with the
definitions made in \cite{kristin}. A discussion of this construction for 
less general smooth varieties (in our terminology, in the case of only 
uniform matroids) can be found in \cite{lars}*{section 3}.

\begin{definition}
Let $f:X\rightarrow Y$ be a morphism of smooth tropical cycles. 
We define the pull-back of a cycle $C$ in $Y$ to be 
\[
  f^{\ast}C:=\pi_{\ast}(\Gamma_f\cdot (X\times C)),
\] 
where $\pi:X\times Y\rightarrow X$ is the projection to the first factor and $\Gamma_f$ is the graph of $f$ (that means $\Gamma_f:={\gamma_f}_{\ast}(X)$, with $\gamma_f:X\rightarrow X\times Y,\  x\mapsto (x,f(x))$.
\end{definition}

Note that here $\Gamma_f\cdot (X\times C)$ is an intersection product of cycles
in $X \times Y$, which is smooth by our assumptions. By definition, we see
that the codimension of $C$ in $Y$ equals the codimension of $f^* C$ in $X$ and 
$|f^* C| \subseteq f^{-1} |C|$. Moreover, we obviously have
$f^*(C + C') = f^* C + f^* C'$.

\begin{example}
\label{expull-back}
  Let us give some examples.
  \begin{enumerate}
    \item
      Let $f:X\rightarrow Y$ be a morphism of smooth tropical cycles.
      Then $f^* Y = X$. This follows easily from $\pi_* (\Gamma_f) = X$.
    \item
      Now we assume additionally that $C = \varphi_1 \cdots \varphi_l \cdot Y$
      is a subcycle of $Y$ cut out by some functions. Then we have
      \[
        f^* C = f^* \varphi_1 \cdots f^* \varphi_l \cdot X.
      \]
      Indeed, if we denote the two projections of $X \times Y$ by $\pi_X$ and 
      $\pi_Y$, then by definition the function $\pi_Y^* \varphi_1$ agrees
      on $\Gamma_f$ with the function $\pi_X^* f^* \varphi_1$ and the above 
      equation follows from projection formula.
    \item
      Let $\id : X \rightarrow X$ be the identity morphism. Then $\Gamma_{\id} = \Delta_X$,
      and we conclude $\id^* C = X \cdot C = C$ for all subcycles $C$ of $X$. 
    \item
      Let $p : X \times Y \rightarrow Y$ be a projection. Then 
      $\Gamma_p = X \times \Delta_Y$, and it follows easily that
      $p^* C = X \times C$ for all subcycles $C$ of $Y$.   
  \end{enumerate}
\end{example}

Our next goal is to prove the following properties of pull-backs:

\begin{theorem}
\label{pull-backproperties}
  Let $X$, $Y$ and $Z$ be smooth tropical varieties and let
  $f : X \rightarrow Y$ and $g : Y \rightarrow Z$ be two morphisms.
  Let $C, C'$ be two cycles in $Y$, $D$ a cycle in $X$ and
  $E$ a cycle in $Z$. 
  Then the following holds:
  \begin{enumerate}
    \item
      $C \cdot f_* D = f_*(f^* C \cdot D)$
    \item
      $f^*(C \cdot C') = f^* C \cdot f^* C'$
    \item
      $(g \circ f)^* E = f^* g^* E$
  \end{enumerate}
\end{theorem}

In a first step we prove the theorem for matroid varieties $X,Y,Z$. We need the following lemma: 

\begin{lemma} \label{help}
    Let $f : X \rightarrow Y$ be a morphism between matroid varieties.
    Then we have 
    \begin{equation} \label{A}
      (\{x\} \times Y) \cdot \Gamma_f = \{(x, f(x))\}
    \end{equation}
    for each point $x$ of $X$. \\
    Let $g : Y \rightarrow Z$ be another morphism of matroid varieties and set
    $\Phi : X \rightarrow X \times Y \times Z$, 
    $x \mapsto (x, f(x), g(f(x)))$. Then we have
    \begin{equation} \label{B}
      \Phi_* X = (\Gamma_f \times Z) \cdot (X \times \Gamma_g).
    \end{equation}
    Analogously, if $h : X \rightarrow Z$ is another morphism of matroid varieties and we
    set $\Phi : X \rightarrow X \times Y \times Z$, 
    $x \mapsto (x, f(x), h(x))$, then we have
    \begin{equation} \label{B'}
      \Phi_* X = (\Gamma_f \times Z) \cdot (\Gamma_h \times Y).
    \end{equation}
    Note that, by abuse of notation with regard to the order
    of the factors, $(\Gamma_h \times Y)$ sits
    inside $X \times Y \times Z$. 
\end{lemma}

\begin{proof}
  We start with equation \eqref{A}. It is obvious that both sides
  are supported on the point $(x, f(x))$, so it suffices to check that
  the degree on the left hand side is $1$. To do this, we can
  assume that $x = 0$ is the origin and is cut out by rational functions on $X$.
  Then $\{x\} \times Y$ is cut out by the pull backs of these functions
  and the projection formula proves the claim. \\
  For equation \eqref{B}, we also start by noting that the supports on both
  sides must be equal. This follows from the fact that
  $\Phi_* X$ is irreducible 
  and the support of the right hand side
  is obviously contained in $|\Phi_* X|$.
  So again, both sides can only differ by a global factor.
  It is easy to see that this factor is indeed $1$: 
  For example, we can intersect both sides with $\{x\} \times Y \times Z$,
  where $x \in |X|$ is any point. Using equation
  \eqref{A} and part (9) of theorem \ref{properties1} it follows that we
  get $1 \cdot\{(x, f(x), g(f(x)))\}$
  on both sides. \\
  Equation \eqref{B'} can be proven completely analogously.
\end{proof}

\begin{proof}[Proof of theorem \ref{pull-backproperties} for matroid varieties $X,Y,Z$]
We give (rather short) proofs of the three properties if $X,Y,Z$ are matroid varieties. 
We skip the details
of a couple of straightforward computations which can be found in more
details in \cite{lars}. In what follows, $\pi := \pi_X$ 
denotes the projection
of a product of $X$, $Y$ and $Z$ to the factor $X$.

To prove (1), it essentially suffices to show $(f \times \id)^* \Delta_Y = \Gamma_f$,
where $f \times \id : X \times Y \rightarrow Y \times Y$.
Using this, a straightforward computation shows
\[
  C \cdot f_* D = f_* \pi_* (\Gamma_f \cdot D \times C) 
                = f_* (f^* C \cdot D).
\]
The equation 
$(f \times \id)^* \Delta_Y = \Gamma_f$ is clear set-theoretically and
the equality of weights can be checked using the first equation of 
lemma \ref{help} and part (2) of example \ref{expull-back}.

To prove (2), another computation shows
\[
  f^*(C \cdot C') = 
    \pi_* ((\Gamma_f \times Y) \cdot (X \times \Gamma_{\id_Y}) 
                  \cdot (X \times C \times C'))
\]
and 
\[
  f^* C \cdot f^* C' = 
    \pi_* (\pi_{1,2}^{\ast}\Gamma_f \cdot \pi_{1,3}^{\ast}\Gamma_f
                  \cdot (X \times C \times C')),
\]
with $\pi_{1,i}:X\times Y\times Y\rightarrow, (x,y_1,y_2)\mapsto (x,y_i)$.
Using both the second and third equality of lemma \ref{help} (with $h=f$ and $g=\id$) together with part (4) of example \ref{expull-back}, we see that 
both terms coincide.

To prove (3), we compute easily
\[
  (g \circ f)^* E = 
     \pi_* ( \Phi_* X \cdot (X \times Y \times E)),
\]
where $\Phi : X \rightarrow X \times Y \times Z$ maps $x$ to $(x, f(x), g(f(x)))$,
and
\[
  f^*g^* E = 
    \pi_* ( (\Gamma_f \times Z) \cdot (X \times \Gamma_g)
                  \cdot (X \times Y \times E) ).
\]
Using the second equation of lemma \ref{help} again, the claim follows.
\end{proof}

In order to extend this proof to arbitrary smooth varieties we need another technical proposition:

\begin{proposition}
\label{pull-back}
Consider the following commutative diagram of tropical morphisms.
\[
  \begin{CD}
    \trop(M)    @>g>>     \trop(N)\\
    @V{\qqq}VV @VV{\jj}V\\
    \trop(M)/L @>f>> \trop(N)/K
  \end{CD}
\]
Here $L,K$ are lineality spaces and $\qqq, \jj$ are the respective quotient maps.
Then the following equality holds: 
\[
  \qqq^{-1}f^{\ast}C = g^{\ast}\jj^{-1}C,
\]  
or equivalently
\[
  f^{\ast}C=(g^{\ast}\jj^{-1}C)/L.
\]  
\end{proposition}

\begin{proof}
Since the cycles $\Gamma_f$ and $\Gamma_{g}$ carry only trivial weights, the equality
\[|(\qqq\times\jj)^{-1}\Gamma_f|=\{(x,y):x\in|\trop(M)|,\ \jj(y)=f\circ\qqq(x)\}=|(\id\times\jj)^{-1}(\id\times\jj)_{\ast}(\Gamma_{g})|\]
implies the equality of cycles \[(\qqq\times\jj)^{-1}\Gamma_f=(\id\times\jj)^{-1}(\id\times\jj)_{\ast}(\Gamma_{g}).\] Let $\pi:\trop(M)/L\times\trop(N)/K\rightarrow \trop(M)/L$ and $\widetilde{\pi}:\trop(M)\times\trop(N)\rightarrow\trop(M)$ projections to the first factor. It follows from the above equality that
\begin{eqnarray*}
f^{\ast}C &=& \pi_{\ast} (((\qqq\times\jj)^{-1}\Gamma_f\cdot (\trop(M)\times\jj^{-1}C))/L\times K) \\
&=& \pi_{\ast} (((\id\times\jj)^{-1}(\id\times\jj)_{\ast}(\Gamma_{g})\cdot (\trop(M)\times\jj^{-1}C))/L\times K).
\end{eqnarray*}
Aplying lemma \ref{qminus1} to the quotient map $(\id\times\jj)$, we see that the above is equal to
\begin{eqnarray*}
&& \pi_{\ast} (((\id\times\jj)^{-1}(\id\times\jj)_{\ast}(\Gamma_{g}\cdot (\trop(M)\times\jj^{-1}C)))/L\times K)\\
&=& \pi_{\ast} ((\id\times\jj)_{\ast}(\Gamma_{g}\cdot(\trop(M)\times\jj^{-1}C))/L\times \{0\})\\
&=& (\widetilde{\pi}_{\ast}(\Gamma_{g}\cdot (\trop(M)\times\jj^{-1}C))/L \\
&=& (g^{\ast}\jj^{-1}C)/L.
\end{eqnarray*}
\end{proof}

\begin{proof}[Proof of theorem \ref{pull-backproperties} for smooth cycles $X,Y,Z$]
We have already proved the claim for matroid varieties $X$, $Y$ and $Z$. 
Using proposition \ref{pull-back}, we see that theorem \ref{pull-backproperties} also holds if $X,Y,Z$ are quotients of matroid varieties by lineality spaces. 
Moreover, as all constructions are based on intersection products and therefore
are defined locally, the statements hold in fact for all smooth varieties
(in our sense).
\end{proof}

\begin{remark}
Let $\trop(M)$ be a matroid variety with lineality space $L$ and let
$\qqq : \trop(M) \rightarrow \trop(M)/L$ be the quotient map. Then the pull-back
$\qqq^* (C)$ coincides with $\qqq^{-1}(C)$ as defined previously. This is a direct consequence
of proposition \ref{pull-back}. 
\end{remark}

\begin{remark}
\label{f_*f^*}
  Let $f : X \rightarrow Y$ be a morphism of smooth tropical varieties
  such that $f_*(X) = Y$. Then it follows from 
  the first part of theorem \ref{pull-backproperties}
  that $f_*f^*(C) = C$ holds for any cycle $C$ in $Y$.
\end{remark}

We now come back to the meaning of pull-backs in the case of modifications. 

\begin{lemma}
\label{pull-backmodification}
Let $Q$, $M$ and $N$ be (loopfree) matroids and let $e$ be an element in $Q$ 
(which is not a coloop) such that $Q \setminus e = M$, $Q/e = N$.
Consider the corresponding modification $\pi : \trop(Q) \rightarrow\trop(M)$ 
and let $\varphi$ be the modification function on $\trop(M)$ (as described in 
proposition \ref{modification}). 
For any subcycle $C$ of $\trop(M)$, let $\widetilde{C}$ be the modification of
$C$ by $\varphi$. Then the equality
\[
  \widetilde{C} = \pi^{\ast}C
\] 
holds.
\end{lemma}

\begin{tabular}{p{6cm}p{1cm}p{4cm}}
\parbox[c]{1em}{\includegraphics[scale=0.26]{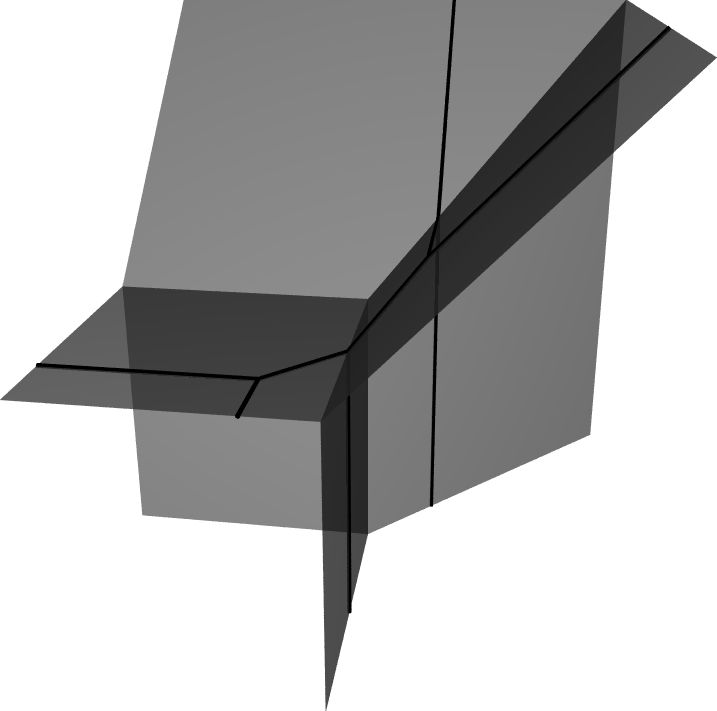}}
&
$\stackrel{\pi}{\longrightarrow}$ &
\parbox[c]{1em}{\includegraphics[scale=0.2]{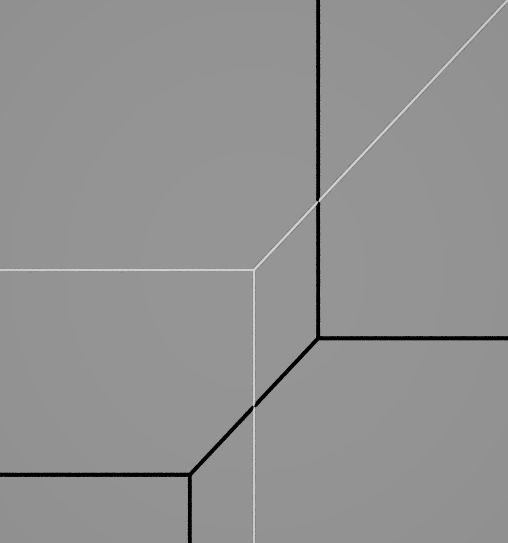}}\\
\multicolumn{3}{c}{Modification of a cycle and its intersection with $\trop(N)$}
\end{tabular}

\begin{proof}
The modification of $C$ along the function $\varphi$ is the (uniquely determined) cycle $\widetilde{C}$ in $\trop(Q)$ satisfying $\pi_{\ast}\widetilde{C}=C$ and 
$\widetilde{C}^{\cap e} = \varphi\cdot C$. We show that $\pi^{\ast} C$ fulfils those two conditions: As $\pi_{\ast}\trop(Q)=\trop(M)$, the first equality follows from
remark \ref{f_*f^*}. 
As for $(\pi^{\ast} C)^{\cap e} = \varphi \cdot C$, we first pick a
real constant $r$ large enough such that 
$\trop(N) = \pi_* (\max\{x_e, -r\} \cdot \trop(Q))$ 
holds.
Applying theorem \ref{pull-backproperties} again provides
\[
  \varphi\cdot C  
    = C\cdot \trop(N)
    = \pi_{\ast}(\pi^{\ast} C \cdot \max\{x_e,-r\} \cdot \trop(Q))
\] 
It follows that $\varphi\cdot C=(\pi^{\ast} C)^{\cap e}$.
\end{proof}

Applying this lemma to a whole series of modifications, we get the following corollary.

\begin{corollary}
  Let $\trop(Q)$ and $\trop(M)$ be matroid varieties such that $Q \setminus R = M$
  (for suitable $R$) and choose a series of matroid modifications
  \[
    \trop(Q) = \trop(M_0) \stackrel{\pi_1}{\rightarrow} 
               \trop(M_1) \stackrel{\pi_2}{\rightarrow} 
               \ldots     \stackrel{\pi_n}{\rightarrow} 
               \trop(M_n) = \trop(M).
  \]
  Let $C$ be a cycle in $\trop(M)$ and let $\widetilde{C}$ be its repeated
  modification along $\pi_n, \ldots, \pi_2, \pi_1$. Then $\widetilde{C}$ is
  in fact independent of the chosen series of modifications.
  
  Moreover, let $\trop(N) \subseteq \trop(M)$ be two matroid varieties
  and let $Q$ be the matroid such that $Q \setminus R = M$ and $Q/R = N$ (cf.\ proof of proposition \ref{quotient}).
  Let $C$ be any cycle in $\trop(M)$. 
  Then the intersection product $\trop(N) \cdot C$ can be computed as
  $(\pi^*C)^{\cap R}$, where $\pi : \trop(Q) \rightarrow \trop(M)$.
  In other words, we get $\trop(N) \cdot C$ by performing a series of 
  modifications that lift $C$ to a cycle in $\trop(Q)$, 
  and then intersecting with a boundary part.
\end{corollary}

Another important consequence of lemma \ref{pull-backmodification} is that
we can now prove that our intersection product coincides with the definitions 
made in \cite{kristin}.

\begin{theorem} \label{compkristin}
  Let $\trop(M)$ be a matroid variety and let $C,D$ be two cycles in $\trop(M)$.
  We denote by $C.D$ the recursive intersection product defined in 
  \cite{kristin}*{definition 3.6}. Then this intersection product coincides with
  the one defined in definition \ref{intersectionproduct}, i.e.\
  \[
    C.D = C \cdot D.
  \]
\end{theorem}

\begin{proof}
  The intersection product $C.D$ of \cite{kristin}*{definition 3.6} is defined
  recursively via modifications. Finally, the recursion uses the known intersection
  product on $\R^n$. As our definition gives back the same product on $\R^n$, we
  have agreement here. It remains to check that our definition satisfies the
  same recursion formula given by
  \[
    C.D = \pi^*(\pi_* C . \pi_* D) + \pi^*\pi_* C.\Delta_D
          + \Delta_C.\pi^*\pi_* D + \Delta_C.\Delta_D,
  \]
  where $\pi : \trop(M) \rightarrow \trop(M\setminus e)$ is a modification and
  $\Delta_C = C - \pi^*\pi_* C$ resp.\ $\Delta_D = D - \pi^*\pi_* D$. 
  Note that in \cite{kristin} $\pi^* E$ is defined as the (restricted) modification of
  $E$, but by lemma \ref{pull-backmodification} we know that we can also use our 
  pull-back definition instead. Writing $C = \pi^*\pi_* C + \Delta_C$
  and $D = \pi^*\pi_* D + \Delta_D$ we get
  \[
    C \cdot D = \pi^*\pi_* C \cdot \pi^*\pi_* D + \pi^*\pi_* C \cdot \Delta_D
          + \Delta_C \cdot \pi^*\pi_* D + \Delta_C \cdot \Delta_D,
  \]
  noting that the first term equals $\pi^*\pi_* C \cdot \pi^*\pi_* D
  = \pi^*(\pi_* C \cdot \pi_* D)$ by theorem \ref{pull-backproperties}
  property (2). So our intersection product satisfies the same recursion formula
  and therefore the definitions agree.
\end{proof}

\section{Rational equivalence on matroid varieties}

Let $C$ be a cycle in $\trop(M)/L$. 
Then by contracting all bounded parts of $C$ to the origin, we
get the so-called \emph{recession cycle} $\delta(X)$ of $X$ (cf.\
\cite{AR2}*{definition 8}). As a set, $|\delta(X)|$ is the limit 
of $t \cdot |C|$ when $t$ goes to zero. The aim of this section is to show
that $C$ is \emph{rationally equivalent} to $\delta(C)$. In this context,
rational equivalence in $\trop(M)/L$ is generated by those cycles which 
are push-forwards along some tropical morphism $f : A  \rightarrow \trop(M)/L$
of a cycle $\varphi \cdot A$, where $\varphi$ is a bounded function
(cf.\ \cite{AR2}*{definition 1}).
Note that, by definition, if $C \sim 0$ in the ambient
space $X$, then $C \sim 0$ also holds in any larger ambient space $Y \supseteq X$.
Our first statement is again concerned with dividing out a lineality space. 

\begin{center}
\includegraphics[scale=0.3]{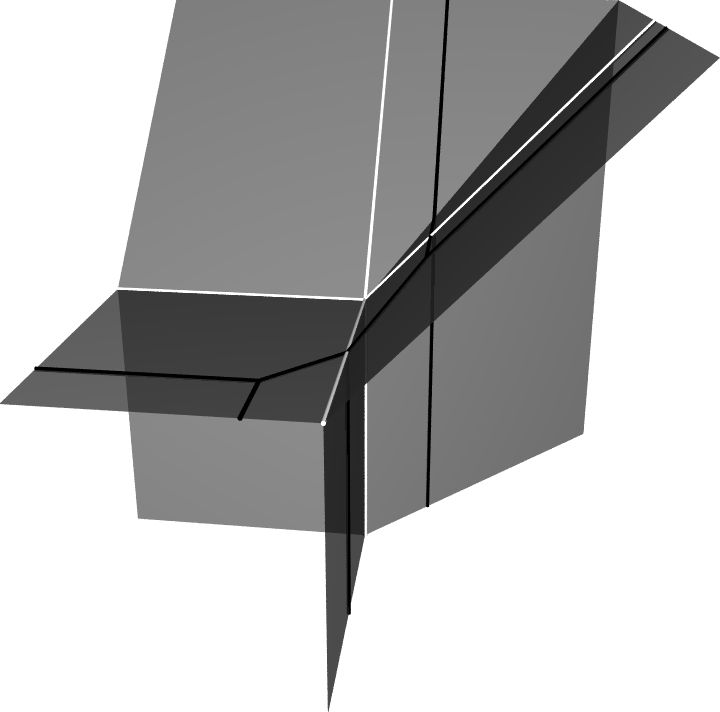}
\end{center}
\begin{center}
A curve on $\trop(U_{3,4})/L$ and its recession cycle.
\end{center}

\begin{proposition}
\label{ratequivlinspace}
Let $X$ be a cycle with lineality space $L$. Let $C$ be a subcycle of $X$ also having lineality space $L$. Then $C$ is rationally equivalent to zero on $X$ if and only if $C/L$ is rationally equivalent to zero on $X/L$.
\end{proposition}

\begin{proof}
As $X$ is isomorphic to $X/L\times L$, it suffices to show that $C/L\times L$ is rationally equivalent to zero on $X/L\times L$ if and only if $C/L$ is rationally equivalent to zero on $X/L$. The if-implication was proved in \cite[lemma 2 (a)]{AR2}. So let us assume that $C/L\times L$ is rationally equivalent to zero. That means by definition that there are a morphism $f:A\rightarrow X/L\times L$ and a bounded rational function $\varphi$ on $A$ such that $f_{\ast}(\varphi\cdot A)=C/L\times L$. Let $\pi_{X/L}:X/L\times L\rightarrow X/L$ and $\pi_L:X/L\times L\rightarrow L$ be projections to the respective factor. We choose rational function on $L$ such that $\psi_1\cdots\psi_{\dim L}\cdot L=\{0\}$. Now we just
replace $A$ by 
$A' := f^{\ast}\pi_L^{\ast}\Psi_1\cdots f^{\ast}\pi_L^{\ast}\Psi_{\dim L}\cdot A$
and check by projection formula that 
$(\pi_{X/L}\circ f)_{\ast}(\varphi \cdot A') = C/L$ holds.
\end{proof}

\begin{remark}
\label{easyremark}
Note that on matroid varieties modulo lineality spaces $\trop(M)/L$, 
intersection products
and pull-backs of cycles are compatible with rational equivalence.
In other words, if $C$ and $C'$ are cycles in $\trop(M)/L$ with 
$C \sim C'$, then also $f^* C \sim f^* C'$ and $C \cdot D \sim C' \cdot D$
for any morphism $f : \trop(N)/K \rightarrow \trop(M)/L$ and any third
cycle $D$ in $\trop(M)/L$.
This follows from the fact that cross products, intersections with rational 
functions as well as push-forwards are compatible with rational equivalence 
(cf.\ \cite{AR2}*{lemma 2}), and the previous proposition.
\end{remark}

In the following, if $a$ is an element of the matroid $M$, we
denote the corresponding projection by
$\pi_a : \trop(M) \rightarrow \trop(M \setminus a)$.
Furthermore, if $\F$ is a chain of flats in $M$, then $\F\setminus a$ denotes the chain of flats in $M\setminus a$ 
obtained by intersecting each flat of $\F$ with $E(M)\setminus a$.

In order to prove that every cycle in $\trop(M)/L$ is rationally equivalent to its recession cycle, we need the following lemmas: 

\begin{lemma}
\label{pia1}
Let $a,b\in E(M)$ be no coloops and assume that $\{b\}$ is a flat. 
Let $C$ be a subcycle of $\trop(M)$ with ${\pi_a}_{\ast} C=0$. Then ${\pi_a}_{\ast}\pi_b^{\ast}{\pi_b}_{\ast}C=0$. 
\end{lemma}


\begin{proof}
We choose a polyhedral structure $\curlyc$ of $C$ which is compatible with pushing forward (cf.\ \cite[lemma 1.3.4]{disshannes}) such that every cell of $\curlyc$ is contained in a cone of $\bfan$. As ${\pi_a}_{\ast}C=0$ we know that every cell of $\curlyc$ is contained in a cone $\langle\F\rangle$ of $\bfan$ satisfying $F_{i+1}=F_i\cup a$ for some $i$ (as on the other facets, $\pi_a$ is one-to-one
and cannot delete non-zero cells of $\curlyc$). In order to simplify the notations we assume that $b=|E(M)|$. Let $\varphi$ be the piecewise linear function on $\mf(M\setminus b)$ which satisfies for all flats $F$ of $M\setminus b$ that 
\[\varphi(V_F)=\begin{cases} 
-1, & \text{ if } b\in\cl_M(F) \\ 
      0,  & \text{ else} 
\end{cases}.  \]
It follows from proposition \ref{modification} that $\trop(M)$ is the modification of  $\trop(M\setminus b)$ along the rational function $\varphi$. Hence $\pi_b^{\ast}{\pi_b}_{\ast} C$ is the modification of ${\pi_b}_{\ast} C$ along $\varphi$ (cf.\ lemma \ref{pull-backmodification}).
 It is easy to see that $\varphi$ is given on a cone $\langle\G \rangle$ of $\mf(M\setminus b)$ by
\[\varphi_{\mid \langle\G\rangle}(x_1,\ldots,x_{b-1})=x_p, \text{ with } p\in G_{z+1}\setminus G_z \text{ and } z \text{ s.t. } b\in\cl_M(G_{z+1})\setminus \cl_M(G_{z}).\]
We claim that for a chain of flats $\F$ in $M$ satisfying $F_{i+1}=F_i\cup a$ for some $i$, the restriction of $\varphi$ to $\langle \F\setminus b\rangle$ does not depend on $x_a$: Assume the contrary is true; then our description of $\varphi$ implies that \[b\in\cl_M(F_{i+1}\setminus b) \text{ and } b\notin\cl_M(F_{i}\setminus b).\]
Note that $\cl_M(F_{i+1}\setminus b)\subseteq F_{i+1}$ and $\cl_M(F_{i}\setminus b)\subseteq F_{i}$; thus $b\in F_{i+1}$. As $F_{i+1}=F_{i}\cup a$, this implies that $b\in F_{i}$. It follows that $\cl_M(F_{i}\setminus b)=F_{i}\setminus b$. Now $F_{i}$ and $\cl_M(F_{i}\setminus b \cup a)=F_{i+1}$ are both minimal flats containing the flat $F_{i}\setminus b$. But this is a contradiction since $F_{i}\subsetneq F_{i+1}$.

Let $\sigma$ be a maximal cell of $\pi_b^{\ast}{\pi_b}_{\ast} \curlyc$ of the form $(\id\times\varphi)(\pi_b(\tau))$, where $\tau$ is a maximal cell of $\curlyc$. 
We can assume that the restriction of $\pi_a$ to $\sigma$ is injective (otherwise $\pi_a(\sigma)$ does not contribute to the push-forward). Since $\varphi_{\mid\pi_b(\tau)}$ does not depend on the $a$-th coordinate, we can conclude that $\alpha:=\pi_{\{a,b\}}(\sigma)$ has the same dimension as $\sigma$. Let $\sigma_1,\ldots,\sigma_p$ be the cells of $\curlyc$ mapped to $\alpha$ by $\pi_{\{a,b\}}$. As $\pi_b$ is injective on $\sigma_i$, the cell $\sigma_i$ turns into the cell 
\[\widetilde{\sigma_i}:=\{(x_1,\ldots,x_{b-1},\varphi(x_1,\ldots,x_{b-1})): \exists \  x_b: (x_1,\ldots,x_b)\in\sigma_i\}\]
in the cycle $\pi_b^{\ast}{\pi_b}_{\ast}C$. The $\widetilde{\sigma_i}$ are exactly the cells of $\pi_b^{\ast}{\pi_b}_{\ast} \curlyc$ mapped to $\pi_a(\sigma)$ by $\pi_a$. 
Since ${\pi_{\{a,b\}}}_{\ast}C=0$, we can conclude that $\pi_a(\sigma)$ has weight $0$ in ${\pi_a}_{\ast} \pi_b^{\ast}{\pi_b}_{\ast} \curlyc$.\\
Now, the claim follows from the balancing condition.
\end{proof}

\begin{lemma}
\label{pia2}
Let $C$ be a subcycle of a matroid variety $\trop(M)$. Assume that $\trop(M)\neq\R^{|E(M)|}$ and that $\{a\}$ is a flat for every $a\in E(M)$. If ${\pi_a}_{\ast}(C) = 0$ for all $a\in E(M)$ which are not coloops of $M$, then $A=0$.
\end{lemma}
\begin{proof}
We choose a polyhedral structure $\curlyc$ of $C$ such that every cell of $\curlyc$ is contained in a cone of $\bfan$. Let $\F=(\emptyset\subsetneq F_1\subsetneq\ldots\subsetneq F_{\rank(M)-1}\subsetneq E(M))$ be an arbitrary maximal chain of flats of $M$. We choose $i$ such that $|F_{i+1}\setminus F_i|>1$ and $a\in F_{i+1}\setminus F_i$. The maximality of $\F$ implies that $a$ is not a coloop. As $\pi_a$ is generically one-to-one (lemma \ref{geometric}) and its restriction to $\langle\F\rangle$ is injective, ${\pi_a}_{\ast}C=0$ implies that there is no cell $\sigma\in\curlyc$ whose interior is contained in the interior of $\langle\F\rangle$. \\
Now we assume there is a cell $\sigma$ of $\curlyc$ whose interior is contained in the interior of a codimension $1$ cone $\langle\G\rangle$ of
$\bfan$. Let $\F=(\emptyset\subsetneq F_1\subsetneq\ldots\subsetneq F_{\rank(M)-1}\subsetneq E(M))$ be a maximal superchain (of flats) of $\G$. As before we choose $a\in F_{i+1}\setminus F_i$, with $i$ satisfying $|F_{i+1}\setminus F_i|>1$. Only cells of $\curlyc$ contained in $\langle\G\rangle$ or a facet adjacent to $\langle\G\rangle$ can potentially be mapped to $\pi_a(\sigma)$ by $\pi_a$. The first part of the proof thus implies that
\[0=\omega_{{\pi_a}_{\ast}C}(\pi_a(\sigma))=\omega_C(\sigma).\]
Continuing this way, we see that $C=0$.
\end{proof}

\begin{theorem}
\label{ratequivbm}
Every subcycle $C$ of a variety $\trop(M)/L$ is rationally 
equivalent to its recession cycle $\delta(C)$.
\end{theorem}

\begin{proof}
By proposition \ref{ratequivlinspace} it suffices to show the statement for 
matroid varieties $\trop(M)$.

We first consider the case where $\{a\}$ is a flat for every $a\in E(M)$. We use induction on the codimension of $\trop(M)$: The induction start ($\trop(M)=\R^n$) was proved in \cite[theorem 7]{AR2}. We show that $C$ is rationally equivalent on $\trop(M)$ to a fan cycle: After renaming the elements, we can assume that $\{1,\ldots,k\}$ is the subset of elements of $E(M)$ which are not coloops. For $i\in\{1,\ldots,k\}$ we set
\[C_0:=C, \ \ \ C_i:=C_{i-1}-\pi_i^{\ast}({\pi_i}_{\ast}C_{i-1}-\delta({\pi_i}_{\ast}C_{i-1})).\]
By induction ${\pi_i}_{\ast}C_{i-1}$ is rationally equivalent to $\delta({\pi_i}_{\ast}C_{i-1})$. As pulling back preserves rational equivalence, it follows that $C_i$ is rationally equivalent to $C_{i-1}$. We set
\[N_0:=C, \ \ \ N_i:=N_{i-1}-\pi_i^{\ast}{\pi_i}_{\ast}N_{i-1},\]
and
\[ F_0:=0, \ \ \ F_i:=F_{i-1}+\pi_i^{\ast}\delta({\pi_i}_{\ast}N_{i-1}).\]
It is easy to see that for all $i$ the cycle $F_i$ is a fan cycle, $C_i=N_i+F_i$, and ${\pi_i}_{\ast}N_i=0$. Lemma \ref{pia1} implies that ${\pi_i}_{\ast}N_{k}=0$ for all $i$; thus $N_{k}=0$ by lemma \ref{pia2}. Therefore, $C$ is rationally equivalent to the fan cycle $F_{k}$. As $\delta(C)$ is the only fan cycle which is rationally equivalent to $C$ on $\R^n$ \cite[lemma 6, theorem 7]{AR2}, we can conclude $F_{k}=\delta(C)$.\\
The general case follows from the observation that the projection $\pi_R:\trop(M)\rightarrow \trop(M\setminus R)$ is an isomorphism for $R=\cl_M(\{a\})\setminus a$.
\end{proof}

\end {document}